\documentclass{amsart}

\usepackage{amssymb,stmaryrd}
\usepackage{amsmath}
\usepackage{graphicx}
\usepackage{epstopdf}
\newtheorem{theorem}{Theorem}[section]

\newtheorem{lemma}[theorem]{Lemma}
\newtheorem{proposition}[theorem]{Proposition}
\newtheorem{corollary}[theorem]{Corollary}

\theoremstyle{definition}
\newtheorem{definition}[theorem]{Definition}

\theoremstyle{remark}
\newtheorem{remark}[theorem]{Remark}

\begin{document}

\title[ Integral Operators Induced by Harmonic Bergman-Besov kernels]{A Class of Integral Operators Induced by Harmonic Bergman-Besov kernels on Lebesgue Classes}

\author{\"{O}mer Faruk Do\u{g}an}
\address{Department of Mathematics,  Tek$\dot{\hbox{\i}}$rda\u{g} Nam{\i}k Kemal University,
59030 Tek$\dot{\hbox{\i}}$rda\u{g}, Turkey}
\email{ofdogan@nku.edu.tr}



\begin{abstract}
We provide a full characterization in terms of the six parameters involved the boundedness of all standard weighted integral operators induced by harmonic Bergman-Besov kernels acting between different Lebesgue classes with standard weights on the unit ball of $\mathbb{R}^{n}$. These operators in some sense generalize  the harmonic Bergman-Besov projections. To obtain the necessity conditions, we use a technique that heavily depends on the precise inclusion relations between harmonic Bergman-Besov and weighted Bloch spaces on the unit ball. This fruitful technique is new. It has been used first  with holomorphic Bergman-Besov kernels by Kaptano\u{g}lu and \"{U}reyen. Methods of the sufficiency proofs we employ are Schur tests or H\"{o}lder or Minkowski type i\-ne\-qu\-a\-li\-ti\-es which also make use of estimates of Forelli-Rudin type integrals.
\end{abstract}

\date{\today}

\maketitle

\section{Introduction}\label{section-Introduction}

Let $n\geq 2$ be an integer,  $\mathbb{B}$ be the  unit ball and $\mathbb{S}$ be the unit sphere  of $\mathbb{R}^n$. Let $\nu$ and $\sigma$ be the  volume and surface  measures on $\mathbb{B}$ and $\mathbb{S}$ normalized so that $\nu(\mathbb{B})=1$ and $\sigma(\mathbb{S})=1$.
For  $\alpha\in \mathbb{R}$,  define the weighted volume measures $\nu_\alpha$ on $\mathbb{B}$ by
\[
d\nu_\alpha(x)=\frac{1}{V_\alpha} (1-|x|^2)^\alpha d\nu(x).
\]
These measures are finite when $\alpha>-1$ and in this case we choose $V_\alpha$ so that $\nu_\alpha(\mathbb{B})=1$. When $\alpha\leq -1$, we set $V_\alpha=1$.
For $0<p<\infty$, we denote the Lebesgue classes with respect to $\nu_\alpha$ by $L^p_{\alpha}$ and the corresponding norms by $\|\cdot\|_{L^p_{\alpha}}$.

Let $h(\mathbb{B})$ be the space of all complex-valued harmonic functions on $\mathbb{B}$ with the topology of uniform convergence on compact subsets. The space of bounded harmonic functions on $\mathbb{B}$ is denoted by $h^{\infty}$.
For $0<p<\infty$ and $\alpha>-1$, the weighted  harmonic Bergman space $b^p_\alpha$ is defined by $b^p_\alpha=  L^p_\alpha \cap h(\mathbb{B})$ endowed with the norm
$\|\cdot\|_{L^p_{\alpha}}$. When $p=2$, the space $b^2_\alpha$ is a Hilbert space with respect to the inner product $[f,g]_{b^2_\alpha}=\int_{\mathbb{B}}f\overline{g} \, d\nu_{\alpha}(x)$ and for each $x\in \mathbb{B}$, the point evaluation functional $f \to f(x)$  is bounded on
$b^2_\alpha$. Thus, by the Riesz representation theorem, there exists the reproducing kernel $R_\alpha(x,\cdot)$ such that $f(x)=[f,R_\alpha(x,\cdot)]_{b^2_\alpha}$ for every $f\in b^2_\alpha$ and $x\in \mathbb{B}$. The homogeneous expansion of $R_\alpha$ is given in the $\alpha>-1$ part of the formulas (\ref{Rq - Series expansion}) and (\ref{gamma k q-Definition}) below (see \cite{GKU2}, \cite{DS}).

The orthogonal projection  $ Q_\alpha:L^2_\alpha \to b^2_\alpha$ is given by the integral operator
\begin{equation}\label{orthogonal projection}
  Q_\alpha f(x)= \frac{1}{V_\alpha} \int_{\mathbb{B}} R_\alpha(x,y) f(y) (1-|y|^2)^\alpha d\nu(y) \quad (f\in L^2_\alpha ).
\end{equation}
The above integral operator plays a major role in the theory of weighted harmonic
Bergman spaces and the question when $ Q_\alpha:L^p_\beta \to L^p_\beta$  is bounded is studied in many sources such as (\cite[ Lemma 3.3]{CR}, \cite[Theorem 7.3]{DS}, \cite[Theorem 3.1]{JP2}, \cite[Lemma 2.4]{PE}, \cite[Propositions 3.5 and 3.6]{LSR}).

The main purpose of this paper is to determine precisely when the integral operator in (\ref{orthogonal projection}) is bounded with considering all possible generalizations. First we
allow for the exponents and the weights to be different and consider the operator
in (\ref{orthogonal projection}) from $L^p_\alpha$ to $L^q_\beta$. Next we allow the parameters in the integrand in (\ref{orthogonal projection}) to be different.

In addition we also remove the restriction $\alpha > -1$ and allow it to be any real
number. The weighted harmonic Bergman spaces $b^p_\alpha$ initially defined for $\alpha > -1$
can be extended to the whole range $\alpha \in \mathbb{R}$. This is studied in \cite{GKU2} and will be
briefly reviewed in Section \ref{section-Preliminaries}. We call the extended family $b^p_\alpha$ $(\alpha \in \mathbb{R}) $ as harmonic Bergman-Besov spaces and the corresponding reproducing kernels $R_\alpha(x,y)$ $(\alpha \in \mathbb{R})$ as harmonic Bergman-Besov kernels. The homogeneous expansion of  $R_\alpha$ in terms of zonal harmonics have the form
\begin{equation}\label{Rq - Series expansion}
R_\alpha(x,y)=\sum_{k=0}^{\infty} \gamma_k(\alpha) Z_k(x,y) \quad  (\alpha\in \mathbb{R}, \, x,y\in \mathbb{B}),
\end{equation}
where (see \cite[Theorem 3.7]{GKU1}, \cite[Theorem 1.3]{GKU2})
\begin{equation}\label{gamma k q-Definition}
        \gamma_k(\alpha):= \begin{cases}
         \dfrac{(1+n/2+\alpha)_k}{(n/2)_k}, &\text{if $\, \alpha > -(1+n/2)$}; \\
         \noalign{\medskip}
         \dfrac{(k!)^2}{(1-(n/2+\alpha))_k (n/2)_k}, &\text{if $\, \alpha \leq -(1+n/2)$},
\end{cases}
\end{equation}
and $(a)_b$ is the Pochhammer symbol. For definition and details about $Z_k(x,y)$, see \cite[Chapter 5]{ABR}.

Finally, we allow the exponents $p, q$ to be $\infty$. Let $L^{\infty}=L^{\infty}(\nu)$ be the Lebesgue class of all essentially bounded functions on $\mathbb{B}$ with respect to $\nu$. In this case we
have $L^{\infty}(d\nu_{\alpha})=L^{\infty}$ for every $\alpha \in \mathbb{R}$ and because of this we need to use a different
weighted class. For $\alpha\in \mathbb{R}$,  we define
\begin{equation*}
\mathcal{L}^{\infty}_\alpha := \{\varphi \,\, \text{is measurable on} \, \, \mathbb{B}: (1-|x|^2)^{\alpha} \varphi(x) \in L^{\infty} \},
\end{equation*}
so that $\mathcal{L}^{\infty}_0 = L^{\infty}$. The norm on $\mathcal{L}^{\infty}_\alpha$ is
\begin{equation*}
\|\varphi\|_{\mathcal{L}^{\infty}_\alpha} = \|(1-|x|^2)^\alpha \varphi(x)\|_{L^\infty}.
\end{equation*}

We are now ready to state our results. For $b,c \in \mathbb{R}$  define the integral operators $T_{bc}$ and $S_{bc}$ by
\begin{equation*}
T_{bc}\, f(x) = \int_{\mathbb{B}} R_{c} (x,y) \, f(y) (1-|y|^2)^b d\nu(y)
\end{equation*}
and
\begin{equation*}
S_{bc}\, f(x) =  \int_{\mathbb{B}} \big|R_{c} (x,y)\big| \, f(y) (1-|y|^2)^b d\nu(y).
\end{equation*}
 We are interested in determining exactly when the above operators are bounded from $L^p_\alpha$ to $L^q_\beta$.  Our main results are the following seven theorems that describe their boundedness in terms of the six parameters $(b,c,\alpha,\beta,p,q)$ involved. We include the operator $S_{bc}$ because we need operators with positive kernels to apply Schur tests.

  The corresponding  integral operators between different Lebesgue classes on the unit ball of $\mathbb{C}^{n}$ in the holomorphic case are considered  in
several publications, such as \cite{Z}. But a complete investigation of the
weighted integral operators arising from holomorphic Bergman-Besov kernels between different weighted Lebesgue classes is recently concluded in \cite{KU1}. We do not attempt to survey the wide literature on different spaces or on more general domains or on
more general weights.

We first consider the case $1\leq p \leq q<\infty$.
The special case $1\leq p=q<\infty$ and $\alpha=\beta$ with $\alpha\in \mathbb{R}$  is considered earlier in \cite{GKU2}. Notice that, in \cite{GKU2} they used the operators which contain an extra factor $(1-|x|^{2})^{a}$ (clearly it does not change anything on the boundedness of operators) and an extra constraints that $c=b+a$. When $\alpha>-n$, the kernel $R_\alpha(x,y)$ is dominated by $1/[x,y]^{\alpha+n}$ (see Lemma \ref{Lemma-Kernel-Estimate} below). Here and subsequently, $[x,y]$ denotes  $[x,y]=\sqrt{1-2 x\cdot y + |x|^2 |y|^2}$ for $x,y\in \mathbb{B}$. For  the boundedness and the norm  of the integral operators which contain  these dominating terms instead of the kernels and an extra factor $(1-|x|^{2})^{a}$ but only for the restricted case $1\leq p=q<\infty$ and $\alpha=\beta>-1$, see \cite{LZ}. We do not try to estimate the norms of the main operators.
The holomorphic counterparts of our two results below on the boundedness of  integral operators induced by holomorphic Bergman-Besov kernels  appear in \cite{KU1}; also the more restricted kernels and case with $\alpha,\beta>-1$ in \cite{KT}.

\begin{theorem}\label{Theorem-Boundedness of T,S 1.1}
Let $b$ and $c$ be real numbers. Let $1<p\leq q<\infty$ and $\alpha,\beta \in \mathbb{R}$  with $\beta>-1$. The following are equivalent:
\begin{enumerate}
\item[(i)] $T_{bc}$ is bounded from $L^p_\alpha$ to $L^q_\beta$.
\item[(ii)] $S_{bc}$ is bounded from $L^p_\alpha$ to $L^q_\beta$.
\item[(iii)] $\alpha+1<p(b+1)$ and $c\leq b+\dfrac{n+\beta}{q}-\dfrac{n+\alpha}{p}$.
\end{enumerate}
\end{theorem}
We have to treat the case $p=1$ separately.
\begin{theorem}\label{Theorem-Boundedness of T,S 1.2}
Let $b$ and $c$  be real numbers. Let $1=p\leq q<\infty$ and $\alpha,\beta \in \mathbb{R}$  with $\beta>-1$. The following are equivalent:
\begin{enumerate}
\item[(i)] $T_{bc}$ is bounded from $L^1_\alpha$ to $L^q_\beta$.
\item[(ii)] $S_{bc}$ is bounded from $L^1_\alpha$ to $L^q_\beta$.
\item[(iii)] $\alpha<b$ and $c\leq b+\dfrac{n+\beta}{q}-(n+\alpha)$ or $\alpha\leq b$ and $c< b+\dfrac{n+\beta}{q}-(n+\alpha)$
\end{enumerate}
\end{theorem}

Now, we consider the case $1\leq q <p <\infty$.
\begin{theorem}\label{Theorem-Boundedness of T,S 2}
Let $b$ and $c$  be real numbers. Let $1\leq q <p< \infty$ and  $\alpha,\beta \in \mathbb{R}$  with $\beta>-1$. The following are equivalent:
\begin{enumerate}
\item[(i)] $T_{bc}$ is bounded from $L^p_\alpha$ to $L^q_\beta$.
\item[(ii)] $S_{bc}$ is bounded from $L^p_\alpha$ to $L^q_\beta$.
\item[(iii)] $\alpha+1<p(b+1)$ and $c< b+\dfrac{1+\beta}{q}-\dfrac{1+\alpha}{p}$.
\end{enumerate}
\end{theorem}

We consider the case when either $p$ or $q$ is $\infty$ in the following four theorems. The special case $ p=q=\infty$ and $\alpha=\beta$  is considered earlier in \cite{DU1} where they used the operators which contain an extra factor $(1-|x|^{2})^{a}$ and an extra constraints that $c=b+a$.
\begin{theorem}\label{Theorem-Boundedness of T,S 3.1}
Let $b$ and $c$ be real numbers. Let $1<p<\infty$ and $\alpha,\beta \in \mathbb{R}$  with $\beta\geq0$. The following are equivalent:
\begin{enumerate}
\item[(i)] $T_{bc}$ is bounded from $L^p_\alpha$ to $\mathcal{L}^\infty_\beta$.
\item[(ii)] $S_{bc}$ is bounded from $L^p_\alpha$ to $\mathcal{L}^\infty_\beta$.
\item[(iii)] $\alpha+1<p(b+1)$ and $c\leq b+\beta-\dfrac{n+\alpha}{p}$, and the strict inequality holds  when $\beta=0$
\end{enumerate}
\end{theorem}
Again, we have to treat the case $p=1$ separately.
\begin{theorem}\label{Theorem-Boundedness of T,S 3.2}
Let $b$ and $c$  be real numbers. Let  $\alpha,\beta \in \mathbb{R}$  with $\beta\geq0$. The following are equivalent:
\begin{enumerate}
\item[(i)] $T_{bc}$ is bounded from $L^1_\alpha$ to $\mathcal{L}^\infty_\beta$.
\item[(ii)] $S_{bc}$ is bounded from $L^1_\alpha$ to $\mathcal{L}^\infty_\beta$.
\item[(iii)] $\alpha<b$ and $c\leq b+\beta-(n+\alpha)$ or $\alpha\leq b$ and $c< b+\beta-(n+\alpha)$
\end{enumerate}
\end{theorem}
\begin{theorem}\label{Theorem-Boundedness of T,S 3.3}
Let $b$ and $c$ be real numbers. Let $1\leq q<\infty$ and  $\alpha,\beta \in \mathbb{R}$  with $\beta>-1$. The following are equivalent:
\begin{enumerate}
\item[(i)] $T_{bc}$ is bounded from $\mathcal{L}^\infty_\alpha$ to $L^q_\beta$.
\item[(ii)] $S_{bc}$ is bounded from $\mathcal{L}^\infty_\alpha$ to $L^q_\beta$.
\item[(iii)] $\alpha-1<b$ and $c<b+\dfrac{\beta+1}{q}-\alpha$.
\end{enumerate}
\end{theorem}
\begin{theorem}\label{Theorem-Boundedness of T,S 3.4}
Let $b$ and $c$ be real numbers.  Let  $\alpha,\beta \in \mathbb{R}$  with $\beta\geq0$. The following are equivalent:
\begin{enumerate}
\item[(i)] $T_{bc}$ is bounded from $\mathcal{L}^\infty_\alpha$ to $\mathcal{L}^\infty_\beta$.
\item[(ii)] $S_{bc}$ is bounded from $\mathcal{L}^\infty_\alpha$ to $\mathcal{L}^\infty_\beta$.
\item[(iii)]  $\alpha-1<b$ and $c\leq b+\beta-\alpha$, and the strict inequality holds when $\beta=0$.
\end{enumerate}
\end{theorem}
It is clear that
\begin{equation*}
|T_{bc}f(x)|\leq S_{bc}(|f|)(x),
\end{equation*}
so the boundedness of $S_{bc}$ implies the boundedness of $T_{bc}$. Thus it is obvious that (ii) implies (i) in all of our theorems above. So "Necessity" and "Sufficiency" in the proofs refers to the implications (i)$\to$ (iii) and (ii)$\to$ (i), respectively.

\begin{remark}
Note the difference in the conditions on $c$ in parts (iii) of Theorems \ref{Theorem-Boundedness of T,S 1.1}--\ref{Theorem-Boundedness of T,S 3.4}. These conditions are connected with the inclusion relations between harmonic Bergman-Besov and weighted Bloch spaces (see Theorems \ref{Besov-Besov1}, \ref{Besov-Besov2} and  \ref{Besov-Bloch} below).
\end{remark}

\begin{remark}
The conditions $\beta>-1$ when $q<\infty$ and $\beta\geq 0$ when $q=\infty$ in our theorems cannot be removed as we clarify later in Corollary \ref{Remark-T,S with beta bigger than -1}. These constraints are consequences of the fact that $T_{bc}f$ is harmonic on $B$ and $|T_{bc}f|^{q}$ is subharmonic when  $q<\infty$, and by the maximum princible for harmonic functions when  $q=\infty$.
\end{remark}

The sufficiency proofs for all seven theorems are either by Schur tests or by direct
H\"{o}lder or Minkowski type inequalities which also make use of growth rate
estimates of Forelli-Rudin type integrals. The necessity proofs are by an
 original technique that most heavily depends on the precise inclusion relations between harmonic Bergman-Besov and weighted Bloch spaces on $\mathbb{B}$. This technique as many others borrowed from \cite{KU1} and have been modified to our kernels and spaces. We give all the sufficiency and necessity proofs in detail and it makes this paper more  self-contained.

This paper is organized as follows. In Section \ref{section-Preliminaries}, we collect some known facts about the harmonic Bergman-Besov and weighted Bloch spaces. In Section \ref{section-Kernels and Operators 1}, we insert  the main
operators in context and derive their basic properties  which we will need in the sequel.
 The corollary about the conditions $\beta>-1$ when $q<\infty$ and $\beta\geq 0$ when $q=\infty$  is also here. In  Section \ref{Section-Schur test}, we list the some important results that we apply in the proofs. As indicated before,  the proofs of Theorems \ref{Theorem-Boundedness of T,S 1.1}--\ref{Theorem-Boundedness of T,S 3.4} contain different methods which are interesting enough to be stated separately. Thus we prove necessity parts of these theorems  in Section \ref{Section-Necessity Proofs} and the sufficiency parts of these theorems in Section \ref{Section-Sufficiency Proofs}.

\section{ Preliminaries}\label{section-Preliminaries}
In multi-index notation, $m=(m_1,\dots,m_n)$ is an n-tuple of non-negative integers $m_1,\dots,m_n$ and
\begin{equation*}
 \partial^m f= \frac{\partial^{|m|} f}{\partial x_1^{m_1}\cdots\partial x_n^{m_n}}
\end{equation*}
is the usual partial derivative for smooth $f$, where $|m|=m_1+\dots+m_n$.

For $1\leq p\leq \infty$, we  denote the conjugate exponent of $p$ by $p'$. That is, if $1< p<\infty$, then $\frac{1}{p}+\frac{1}{p'}=1$; if $p=1$, then $p'=\infty$ and if $p=\infty$, then $p'=1$.

For two positive expressions $X$ and $Y$, we  write $X\sim Y$ if $X/Y$ is bounded  above and below by some positive constants. We will denote these constants whose exact values are inessential by a generic upper case $C$. We will also write $X\lesssim Y$ to mean $X\leq C Y$.

Now we clarify the notation used  in (\ref{gamma k q-Definition}) at the beginning. The Pochhammer symbol $(a)_b$ is given by
\begin{equation*}
 (a)_b=\frac{\Gamma(a+b)}{\Gamma(a)},
\end{equation*}
when $a$ and $a+b$ are off the pole set $-\mathbb{N}$ of the gamma function $\Gamma$. By the Stirling formula
\begin{equation}\label{Stirling}
  \frac{(a)_c}{(b)_c} \sim c^{a-b}, \quad c\to \infty.
\end{equation}

Let $f\in L^{1}_{0}$. The polar coordinates formula is
\begin{equation*}
\int_{\mathbb{B}}f(x) \ d\nu(x)=n\int_{0}^{1} \epsilon^{n-1}\int_{\mathbb{S}}f(\epsilon\zeta) \ d\sigma(\zeta) \ d\epsilon,
\end{equation*}
in which $x=\epsilon\zeta$ with $\epsilon>0$ and $\zeta \in \mathbb{S}$.

As  mentioned in the introduction, for $x,y\in \mathbb{B}$, we will use the notation
\begin{equation*}
  [x,y]=\sqrt{1-2 x\cdot y + |x|^2 |y|^2},
\end{equation*}
where $x\cdot y$ denotes the inner product of $x$ and $y$ in $\mathbb{R}^n$. It is elementary to show that the equalities
\begin{equation*}
  [x,y] = \Big| |y|x - \frac{y}{|y|}\Big| = \Big| |x|y - \frac{x}{|x|}\Big|,
\end{equation*}
hold for every nonzero $x,y$. Note that $0< 1-|x||y|\leq [x,y]\leq 1+|x||y|<2$ for $x,y \in \mathbb{B}$. Further,  we have $[x,\zeta]=|x-\zeta|$ when $y=\zeta\in \mathbb{S}$.

We show an integral inner product on a function space $A$  by $[\cdot,\cdot]_{A}$.

\subsection{Harmonic Bergman-Besov and Weighted Bloch Spaces}
It is well-known that $f\in h(\mathbb{B})$ has a homogeneous expansion
$f=\sum_{k=0}^{\infty} f_k$, where $f_k$ is a homogeneous harmonic polynomial of degree $k$, the series absolutely and uniformly converges on compact subsets of $\mathbb{B}$ (see \cite{ABR}).

The weighted harmonic Bergman spaces $b^p_\alpha$$(\alpha>-1)$ can be extended  to all $\alpha \in \mathbb{R}$. Thus, we resort to derivatives. For $\alpha\in\mathbb{R}$ and $0<p<\infty$, let $N$ be a non-negative integer such that
$\alpha+pN>-1$. The harmonic Bergman-Besov space $b^p_\alpha$ consists of all $f\in h(\mathbb{B})$ such that
\[
(1-|x|^2)^N \partial^m f \in L^p_\alpha,
\]
for every multi-index $m$ with $|m|=N$.

Roughly speaking the $``p=\infty"$ case of Bergman-Besov spaces $b^p_\alpha$ is the family of weighted Bloch spaces $b^\infty_\alpha$. Let $\alpha\in \mathbb{R}$. Pick a non-negative integer $N$ such that $\alpha+N>0$. The weighted harmonic Bloch space $b^\infty_\alpha$ consists of all $f\in h(\mathbb{B})$ such that
\[
 (1-|x|^2)^{N} \partial^m f \in \mathcal{L}^\infty_\alpha,
\]
for every multi-index $m$ with $|m|=N$. We mention one special case. When $\alpha=0$ taking $N=1$ shows
\[
b^\infty_0=\Big\{f\in h(\mathbb{B}): \sup_{x\in\mathbb{B}}\, (1-|x|^2)|\nabla f(x)| <\infty\Big\}.
\]
This is the most studied member of the family.

Partial derivatives are not convenient in studying the spaces of interest in this work and it is more advantageous to use certain radial differential operators $D^t_s: h(\mathbb{B}) \to h(\mathbb{B})$, $(s,t \in \mathbb{R})$ introduced in \cite{GKU1} and \cite{GKU2} that
are compatible with the kernels.

Before going to the definition, note that for every $\alpha\in \mathbb{R}$ we have $\gamma_{0} (\alpha)=1$, and therefore
\begin{equation}\label{Rq(x,0)}
R_\alpha(x,0)=R_\alpha(0,y)=1, \quad (x,y\in \mathbb{B}, \alpha\in \mathbb{R}).
\end{equation}
Checking the two cases in (\ref{gamma k q-Definition}), we have by (\ref{Stirling})
\begin{equation}\label{gamma-k-asymptotic}
\gamma_k(\alpha) \sim k^{1+\alpha} \quad (k\to \infty).
\end{equation}

\begin{definition}
Let $f=\sum_{k=0}^\infty f_k\in h(\mathbb{B})$ be given by its homogeneous expansion. For $s,t\in\mathbb{R}$ we define  $D_s^t : h(\mathbb{B}) \to h(\mathbb{B})$ by
\begin{equation}\label{Define-Dst}
  D_s^t f := \sum_{k=0}^\infty \frac{\gamma_k(s+t)}{\gamma_k(s)} \, f_k.
\end{equation}
\end{definition}
By (\ref{gamma-k-asymptotic}), $\gamma_k(s+t)/\gamma_k(s) \sim k^t$ for any $s,t$. So $D_s^t$ is a differential operator of order $t$. For every $s\in \mathbb{R}$, $D_s^0=I$, the identity. An important property of $D^t_s$ is that it is invertible with two-sided inverse $D_{s+t}^{-t}$:
\begin{equation}\label{inverse of Dst}
D^{-t}_{s+t} D^t_s = D^t_s D^{-t}_{s+t} = I,
\end{equation}
which follows from the additive property
\begin{equation}\label{Additive-Dst}
D_{s+t}^{z} D_s^t = D_s^{z+t}.
\end{equation}
Thus any $D_s^t$  maps $h(\mathbb{B})$ onto itself.
Then for every $s,t \in \mathbb{R}$, the map $D^t_s: h(\mathbb{B})\to h(\mathbb{B})$ is continuous. For a proof see \cite[ Theorem 3.2]{GKU2}

The parameter $s$ plays a minor role. It is used to have the precise relation
\begin{equation}\label{Dst - Rs}
D_s^t R_s(x,y)=R_{s+t}(x,y),
\end{equation}
where differentiation is performed on either of the variables $x$ or $y$ and by symmetry
it does not matter which.

One of the most important properties about the operators $D^t_s$ is that it allows us to pass from one Bergman-Besov (or Bloch) space to another. More precisely, we have the following results.

\begin{lemma}\label{Apply-Dst}
Let $0<p<\infty$ and $\alpha,s,t\in \mathbb{R}$.
\begin{enumerate}
  \item[(i)] The map $D^t_s:b^p_\alpha \to b^p_{\alpha+pt}$ is an isomorphism.
  \item[(ii)] The map $D^t_s:b^\infty_\alpha \to b^\infty_{\alpha+t}$ is an isomorphism.
\end{enumerate}
\end{lemma}

For a proof of part (i) of the above lemma see \cite[Corollary 9.2]{GKU2} when $1\leq p<\infty$ and \cite[Proposition 4.7]{DOG} when $0<p<1$. For part (ii) see \cite[Proposition 4.6]{DU1}.

Consider the linear transformation $I_{s}^{t}$ defined for $f\in h(\mathbb{B})$ by
\begin{equation*}
  I^t_s f(x) := (1-|x|^2)^t D^t_s f(x).
\end{equation*}
The  harmonic Bergman-Besov space and Bloch space can equivalently be defined
by using the operators  $ D^t_s $.

\begin{definition}\label{definition of the h B-B space}
For $0<p<\infty$ and $\alpha \in \mathbb{R}$, we define the harmonic Bergman-Besov space $b^p_\alpha$ to consists of all $f\in h(\mathbb{B})$ for which $ I^t_s f$
belongs to  $L^p_\alpha$ for some  $s,t$ satisfying (see \cite{GKU2} when $1\leq p<\infty$, and \cite{DOG} when $0<p<1$)
\begin{equation}\label{alpha+pt}
 \alpha+pt>-1.
 \end{equation}
The quantity
\[
\|f\|^p_{b^p_\alpha} = \| I^t_s f\|^p_{L^p_\alpha}=c_\alpha \int_{\mathbb{B}} |D^t_s f(x)|^p (1-|x|^2)^{\alpha+pt} d\nu(x) <\infty
\]
defines a norm (quasinorm when $0<p<1$) on $b^p_\alpha$ for any such $s,t$.
\end{definition}

When $\alpha>-1$, one can choose $N=0$ and the resulting space is weighted harmonic  Bergman space.  when $\alpha=-n$, the measure $d\nu_{-n}$ is M\"{o}bius invariant and the spaces $b^p_{-n}$ are called harmonic Besov spaces by many authors. In particular, the $b^2_{-1}$ is the harmonic Hardy space and  $b^2_{-n}$ is the harmonic Dirichlet space.
\begin{definition}\label{definition of the h B space}
For $\alpha \in \mathbb{R}$, we define the harmonic Bloch space $b^\infty_\alpha$ to consists of all $f\in h(\mathbb{B})$  for which $ I^t_s f$ belongs to  $\mathcal{L}^\infty_\alpha$ for some  $s,t$ satisfying (see \cite{DU1})
\begin{equation}\label{alpha+t}
 \alpha+t>0.
\end{equation}
The quantity
\[
\|f\|_{b^\infty_\alpha}=\| I^t_s f\|^p_{L^\infty_\alpha}= \sup_{x\in \mathbb{B}}\, (1-|x|^2)^{\alpha+t} |D^t_s f(x)| <\infty.
\]
defines a norm on $b^\infty_\alpha$ for any such $s,t$.
\end{definition}

\begin{remark}
By now, it is well-known that Definitions \ref{definition of the h B-B space} and \ref{definition of the h B space} are independent of $s,t$ under (\ref{alpha+pt}) and (\ref{alpha+t}), respectively. Moreover, the norm (quasinorm when $0<p<1$) on a given space depends on $s$ and $t$ but this is not mentioned as it is known that every choice of the pair $(s,t)$ leads to an equivalent norm. Thus for a given pair $s,t$,  $ I^t_s$ isometrically imbeds $b^p_\alpha$ into $L^p_\alpha$ if and only if (\ref{alpha+pt}) holds, and $ I^t_s$ isometrically imbeds $b^\infty_\alpha$ into $L^\infty_\alpha$ if and only if (\ref{alpha+t}) holds.
\end{remark}

We turn to  properties and estimates of reproducing kernels. For every $\alpha\in \mathbb{R}$, the series in (\ref{Rq - Series expansion}) absolutely and uniformly converges on $K\times \mathbb{B}$, for any compact subset $K$ of $\mathbb{B}$. Furthermore $R_\alpha(x,y)$ is real-valued, symmetric in the variables $x$ and $y$ and harmonic with respect to each variable.

The $\alpha\geq -1$ part of the the following pointwise estimates for $R_\alpha(x,y)$ and  its partial derivatives are proved in many places including \cite{CKY, JP,R}. For a proof  when $\alpha\in\mathbb{R}$ we refer to \cite[Corollary 7.1]{GKU2}.
\begin{lemma}\label{Lemma-Kernel-Estimate}
Let $\alpha\in\mathbb{R}$ and $m$ be a multi-index. Then for every $x\in \mathbb{B}$, $y\in \overline{\mathbb{B}}$,
\begin{equation*}
\big|(\partial^m R_\alpha)(x,y)\big|
\lesssim\begin{cases}
1,&\text{if $\, \alpha+|m|<-n$};\\
1+\log \dfrac{1}{[x,y]},&\text{if $\, \alpha+|m|=-n$};\\
\dfrac{1}{[x,y]^{n+\alpha+|m|}},&\text{if $\, \alpha+|m|>-n$}.
\end{cases}
\end{equation*}
\end{lemma}
It follows from the above lemma that if $K\subset \mathbb{B}$ is compact and $m$ is a multi-index, then
\begin{equation}\label{Rq-uniformly bounded}
| \partial^m R_\alpha(x,y) | \lesssim 1 \quad ( x\in K, \, y\in \overline{\mathbb{B}}),
\end{equation}
where differentiation is performed in the first variable.

The next lemma shows that the above estimate holds in two directions on the diagonal $x=y$. For a proof see \cite[Proposition 4 (i)]{M} when $\alpha>-1$ and \cite[Lemma 2.9]{DU1} when $\alpha\in\mathbb{R}$.

\begin{lemma}\label{x-equal-y}
Let $\alpha\in\mathbb{R}$. For all $x\in\mathbb{B}$,
\[
R_\alpha(x,x) \sim
\begin{cases}
\dfrac{1}{(1-|x|^2)^{\alpha+n}},&\text{if $\, \alpha>-n$};\\
1+\log \dfrac{1}{1-|x|^2},&\text{if $\, \alpha=-n$};\\
1,&\text{if $\, \alpha<-n$}.
\end{cases}
\]
\end{lemma}

The lemma below is taken from \cite[Lemma 3.2]{DU1} and it shows that if $x$ stays close to $0$, then $R_\alpha(x,y)$ is uniformly away from $0$ for every $y\in \mathbb{B}$.  Recall also that $R_\alpha(0,y)=1$ for every $\alpha\in \mathbb{R}$ and $y\in \mathbb{B}$.

\begin{lemma}\label{Lemma-Stay away from 0}
Let $\alpha\in \mathbb{R}$. There exists $\epsilon>0$ such that for all $|x|<\epsilon$ and for all $y\in \mathbb{B}$, we have $R_\alpha(x,y) \geq 1/2$.
\end{lemma}

For $1\leq p < \infty$, we have bounded projections from the $L^p_\alpha$ onto the $b^p_\alpha$.
\begin{definition}
  For $s\in \mathbb{R}$, the harmonic Bergman-Besov projection is
\begin{equation*}
  Q_s f(x)= \frac{1}{V_{s}}T_{ss} = \int_{\mathbb{B}} R_s(x,y) f(y) d\nu_s(y),
\end{equation*}
for suitable $f$.
\end{definition}

The following two theorems describes the boundedness  of Bergman-Besov projections on $b^p_\alpha$ and $b^\infty_\alpha$  spaces, and are Theorem 1.5 of \cite{GKU2} and Theorem 1.6 of \cite{DU1}, respectively.

\begin{theorem}\label{Theorem-Projection-Besov}
Let $1\leq p < \infty$ and $\alpha, s \in \mathbb{R}$. Then $Q_s : L^p_\alpha \to b^p_\alpha$ is bounded (and onto) if and only if
\begin{equation}\label{Projection-Besov-S}
\alpha+1< p(s+1).
\end{equation}
Given an $s$ satisfying (\ref{Projection-Besov-S}) if $t$ satisfies
\begin{equation}\label{Projection-T}
\alpha+pt>-1,
\end{equation}
then for $f\in b^p_\alpha$, we have
\begin{equation}\label{QsIst=f1}
Q_s I^t_s f = \frac{V_{s+t}}{V_s} f.
\end{equation}
\end{theorem}

\begin{theorem}\label{Theorem-Projection}
Let $\alpha, s \in \mathbb{R}$. Then $Q_s : L^\infty_\alpha \to b^{\infty}_\alpha$ is bounded (and onto) if and only if
\begin{equation}\label{Projection-S}
s>\alpha-1.
\end{equation}
Given an $s$ satisfying (\ref{Projection-S}), if $t$ satisfies
\begin{equation}\label{Projectiont}
\alpha+t>0,
\end{equation}
then for $f\in b^{\infty}_\alpha$, we have
\begin{equation}\label{QsIst=f}
Q_s I^t_s f = \frac{V_{s+t}}{V_s} f.
\end{equation}
\end{theorem}

\section{Properties of the Operators}\label{section-Kernels and Operators 1}
We now formulate the behavior of the operators $T_{bc}$ in many different circumstances. These are adapted from similar results   in  \cite{KU1}. First, we insert some obvious inequalities which will be useful in the proofs. If $a_{1}<a_{2}$, $u>0$, and $v \in \mathbb{R}$, then for $0\leq t<1$,
\begin{equation}\label{obvious-inequalities}
  (1-t^2)^{a_{1}}\leq (1-t^2)^{a_{2}} \qquad \text{and} \qquad (1-t^2)^{u} \big(1+\log (1-t^2)^{-1}\big)^{-v}\lesssim 1.
\end{equation}

The second inequality above leads to an estimate that we need many times.
\begin{lemma}\label{an estimate from calculus}
For $u,v \in \mathbb{R}$,
\begin{equation}\label{int-obvious-inequalities}
\int_{0}^{1} (1-t^{2})^{u}\big(1+\log \frac{1}{1-t^{2}}\big)^{-v} \, dt<\infty
\end{equation}
if $u>-1$ or $u=-1$ and $v>1$, and the integral diverges otherwise.
\end{lemma}
\begin{proof}
The integral have only one singularity at $t = 1$. Polynomial
growth dominates a logarithmic one for $u\neq  -1$. For $u= -1$, we reduce the integral into one studied in calculus after changes of variables and at this time we need $v>1$ for the convergence of the integral.
\end{proof}

We will use the functions
\begin{equation*}
f_{uv}(x)= (1-|x|^{2})^{u}\big(1+\log \frac{1}{1-|x|^{2}}\big)^{-v} \quad (u,v \in \mathbb{R})
\end{equation*}
as test functions to obtain some of the necessary conditions of our theorems from the action of $ T_{bc}$ on them. If we apply Lemma \ref{an estimate from calculus} to the $f_{uv}$, we get the following result.

\begin{lemma}\label{when fuv in Lpq}
For $1\leq p<\infty$, we have $f_{uv}\in L^p_\alpha$ if and only if $\alpha+pu>-1$, or $\alpha+pu=-1$ and $pv>1$. For $p=\infty$, we have $f_{uv}\in \mathcal{L}^{\infty}_{\alpha}$ if and only if $\alpha+u>0$, or $u=-\alpha$ and $v\geq 0$.
\end{lemma}

\begin{lemma}\label{when Tfuv is finite}
If $b+u>-1$ or if $b+u=-1$ and $v>1$, then $T_{bc}f_{uv}$ is a finite positive constant. Otherwise, $T_{bc}f_{uv}(x)=\infty$ for $|x|\leq \epsilon$, where $\epsilon$ is as in Lemma \ref{Lemma-Stay away from 0}.
\end{lemma}
\begin{proof}
If $b+u>-1$ or if $b+u=-1$ and $v>1$, then integrating in polar
coordinates to obtain
\begin{align*}
T_{bc}f_{uv}(x) & = \int_{\mathbb{B}} R_{c}(x,y)\, (1-|y|^{2})^{b+u}\big(1+\log \frac{1}{1-|x|^{2}}\big)^{-v} d\nu(y)\\
               & = \int_0^1 nt^{n-1}(1-t^2)^{b+u} \big(1+\log \frac{1}{1-t^{2}}\big)^{-v}\int_{\mathbb{S}} R_{c} (x,t\zeta)d\sigma(\zeta) dt.
\end{align*}
By the mean-value property the integral over $\mathbb{S}$ is $R_{c}(x,0)$ which is $1$ by (\ref{Rq(x,0)}). Thus,
\begin{align*}
T_{bc}f_{uv}(x) &= \int_0^1 nt^{n-1}(1-t^2)^{b+u} \big(1+\log \frac{1}{1-t^{2}}\big)^{-v} R_{c} (x,0) dt.\\
               &= \int_0^1 nt^{n-1}(1-t^2)^{b+u} \big(1+\log \frac{1}{1-t^{2}}\big)^{-v}dt.
\end{align*}
The  last integral is finite by Lemma \ref{an estimate from calculus}, and then clearly $T_{bc}f_{uv}$ is a constant.

For the other values of the parameters,
\begin{equation*}
  T_{bc}f_{uv}(x) \geq \frac{1}{2} (1-|y|^{2})^{b+u}\big(1+\log \frac{1}{1-|x|^{2}}\big)^{-v} d\nu(y)=\infty
\end{equation*}
by Lemma \ref{Lemma-Stay away from 0} for $|x| < \epsilon$ and Lemma \ref{an estimate from calculus}.
\end{proof}

One can easily compute the adjoint of $T_{bc}$.

\begin{proposition}\label{Adjoint of T}
The formal adjoint  $T^{*}_{bc}:L^{q'}_\beta \to L^{p'}_\alpha$ of the operator $T_{bc}:L^{p}_\alpha \to L^{q}_\beta$ for $1\leq p,q< \infty $ is  $T^{*}_{cb}=(1-|x|^2)^{b-\alpha} T_{\beta c}$.
\end{proposition}
\begin{proof}
Let $f\in L^{p}_\alpha $ and $g\in L^{q}_\beta$. Then by the definition, the real-valuedness and symmetry
in its two variables of $R_{c}(x,y)$ along with Fubini theorem, we obtain
\begin{align*}
[T_{bc}f,g]_{L^{2}_\beta}& =\int_{\mathbb{B}} \int_{\mathbb{B}} R_{c}(x,y)f(x)(1-|x|^{2})^{b}d\nu(x)\overline{g(y)} (1-|y|^{2})^{\beta}d\nu(y)\\
               & =\int_{\mathbb{B}}f(x) \overline{(1-|x|^{2})^{b-\alpha}\int_{\mathbb{B}} R_{c}(x,y)g(y)(1-|y|^{2})^{\beta}d\nu(y)}\\
               &\times(1-|x|^{2})^{\alpha}d\nu(x)\\
                & =\int_{\mathbb{B}}f \, \overline{T^{*}_{bc}}d\nu_{\alpha}=[f,T^{*}_{bc}g]_{L^{2}_\alpha}.
\end{align*}

Hence,
\begin{equation*}
T^{*}_{bc}g(x)=(1-|x|^2)^{b-\alpha}\int_{\mathbb{B}} R_{c}(x,y)g(y)(1-|y|^{2})^{\beta}d\nu(y)
\end{equation*}
\end{proof}
We will use the following simple but very important result.  It must have been known by the experts, even though we could not find a reference in the literature.
\begin{lemma}\label{not whbs}
Let $0<q<\infty$, $\beta\leq-1$ and $f\in h(\mathbb{B})$. If $f\not\equiv0$, then
\begin{equation*}
\int_{\mathbb{B}} |f(x)|^{q}(1-|x|^{2})^{\beta} \, d\nu(x)=\infty.
\end{equation*}
\end{lemma}

\begin{corollary}\label{Remark-T,S with beta bigger than -1}
If $T_{bc}:L^p_\alpha\to L^q_\beta$ is bounded and  $f\in L^p_\alpha$,  then $g=T_{bc}f$ is harmonic on $\mathbb{B}$. If also $q<\infty$, then   $\beta>-1$. Therefore $T_{bc}:L^p_\alpha\to b^q_\beta$  when it is bounded with $\beta>-1$ and $q<\infty$. Moreover, if  $\beta\leq -1$ and $q<\infty$, then  $T_{bc}:L^p_\alpha\to L^q_\beta$ is  not bounded. On the other hand, if $T_{bc}:L^p_\alpha\to L^\infty$ is bounded and  $f\in L^p_\alpha$,  then $g=T_{bc}f\in h^{\infty}$. Finally, If $T_{bc}:L^p_\alpha\to \mathcal{L}^\infty_{\beta}$ is bounded, $f\in L^p_\alpha$, and $\beta>0$  then $g=T_{bc}f\in b^\infty_{\beta}$. Moreover, if  $\beta< 0$, then  $T_{bc}:L^p_\alpha\to \mathcal{L}^\infty_\beta$ is  not bounded.
\end{corollary}

\begin{proof}
That $g$ is harmonic follows, for example, by differentiation under the integral sign, from the fact that $R_{\alpha}(x,y)$ is harmonic in $x$. That $\beta>-1$ when $q<\infty$ follows from Lemma \ref{not whbs}.   For $\beta<0$, $h(\mathbb{B})\cap \mathcal{L}^\infty_\beta$ contains only
$g\equiv 0$  by the maximum principle for harmonic functions.
\end{proof}

\section{Main Tools}\label{Section-Schur test}

Let $(X,\mu)$ and $(Y,\upsilon)$ be $\sigma$-finite measure spaces. Let $K(x,y)$ be a non-negative measurable function on $X\times Y$. Let us denote by G the integral operator with kernel $K$:
\begin{equation*}
Gf(y)=\int_{X} K(x,y) f(x)  d\mu(x).
\end{equation*}
Schur test is a sufficiency condition for the boundedness of $G$ from  $L^{p}(X,\mu)$ to $L^{q}(Y,\upsilon)$.

First we take up the Schur test for the case $1< p\leq q<\infty$. For a proof, see \cite[Theorem 2.1]{O} or \cite[Theorem 1]{Z}.
\begin{theorem}\label{Schur test 1}
Suppose $1<p\leq q<\infty$. Let  $\gamma$ and $\delta$ be two real numbers with $\gamma+\delta=1$. If there exists two strictly positive functions $\phi$ (on $X$) and $\psi$ (on $Y$)   with positive constants $C_{1}$ and $C_{2}$ such that
\begin{equation*}
\int_{X} (K(x,y))^{\gamma p'} (\phi(x))^{p'}  d\mu(x)\leq C_{1}(\psi(y))^{p'},
\end{equation*}
for almost every $y\in Y$ and
\begin{equation*}
\int_{Y} (K(x,y))^{\delta q} (\psi(y))^{q}  d\upsilon(y)\leq C_{2}(\phi(x))^{q},
\end{equation*}
for almost every $x\in X$, then $G$ is bounded from $L^{p}(X,\mu)$ into $L^{q}(Y,\upsilon)$ and the norm of $G$ does not exceed $C_{1}^{1/p'}C_{2}^{1/q}$.
\end{theorem}

We also have the following Schur test for the case $1<q< p<\infty$. For a proof, see \cite[Theorem 1]{G} which also attributes it to \cite{AMSZ}.
\begin{theorem}\label{Schur test 2}
Suppose $1<q< p<\infty$. If there exists two strictly positive functions $\phi$ (on $X$) and $\psi$ (on $Y$) with positive constant $C$ such that
\begin{align*}
&\int_{X} K(x,y) \phi(x)^{p'} d\mu(x)\leq C(\psi(y))^{q'},\\
&\int_{Y} K(x,y) \psi(y)^{q}  d\upsilon(y)\leq C(\phi(x))^{p},
\end{align*}
for almost every $y\in Y$ and $x\in X$, respectively and
\begin{equation*}
\int\int_{X\times Y} K(x,y)\phi(x)^{p'}\psi(y)^{q}   d\mu\times d\upsilon(x,y) \leq C,
\end{equation*}
 then $G$ is bounded from $L^{p}(X,\mu)$ into $L^{q}(Y,\upsilon)$ and the norm of $G$ does not exceed $C$.
\end{theorem}

We also need the following less known Minkowski integral inequality that in effect exchanges the order of integration; for a proof, see  \cite[Theorem 3.3.5]{O1} for example.

\begin{lemma}\label{Minkowski int-inequality}
If $1\leq p\leq \infty$ and   $f(x,y)$ is a measurable function on $X\times Y$, then
\begin{equation*}
\left(\int_{Y}\left(\int_{X} |f(x,y)| d\mu(x)\right)^{p}  d\upsilon(y)\right)^{1/p}\leq \int_{X}\left(\int_{Y} |f(x,y)|^{p} d\upsilon(y)\right)^{1/p}  d\mu(x),
\end{equation*}
with an appropriate interpretation with the $L^{\infty}$ norm when $p=\infty$.
\end{lemma}

The next lemma provides an estimate on weighted integrals of powers of $R_\alpha(x,y)$.
When $\alpha>-1$ and $w>0$, it is proved in \cite[Proposition 8]{M}. For the whole range $\alpha\in \mathbb{R}$ see \cite[Theorem 1.5]{GKU2}.

\begin{lemma}\label{norm-kernel}
Let $\alpha\in \mathbb{R}$, $0<p<\infty$ and $d>-1$. Set $w=p(n+\alpha)-(n+d)$. Then
\[
\int_{\mathbb B}|R_\alpha(x,y)|^p\,(1-|y|^2)^d\,d\nu(y)
\sim\begin{cases}
1,&\text{if $w<0$};\\
\noalign{\medskip}
1+\log\dfrac1{1-|x|^2},&\text{if $w=0$};\\
\noalign{\medskip}
\dfrac1{(1-|x|^2)^w},&\text{if $w>0$}.
\end{cases}
\]
\end{lemma}

Notice that the kernel $R_\alpha(x,y)$ is dominated above by $1/[x,y]^{n+\alpha}$ by taking $|m|=0$ when $\alpha>-n$ in Lemma \ref{Lemma-Kernel-Estimate}. The following integral estimate of these dominating terms will be crucial to the proof our main results. For a proof see \cite[Proposition 2.2]{LS} or \cite[Lemma 4.4]{R}.

\begin{lemma}\label{Integral-[x,y]}
Let $d>-1$ and $s\in \mathbb{R}$. Then
\begin{equation*}
  \int_{\mathbb{B}} \frac{(1-|y|^2)^d}{[x,y]^{n+d+s}} \, d\nu(y) \sim
      \begin{cases}
         1, &\text{if $\, s<0$};\\
         1+\log \dfrac{1}{1-|x|^2}, &\text{if $\, s=0$}; \\
         \dfrac{1}{(1-|x|^2)^s}, &\text{if $\, s > 0$}.
     \end{cases}
\end{equation*}
\end{lemma}

We can push $D^t_s$ into some certain integrals. The following lemma is taken from \cite[Lemma 2.3]{DU1}.

\begin{lemma}\label{Lemma-Push-Dst}
Let $b\in \mathbb{R}$ and $f\in L_b^1$. For every $s,t \in \mathbb{R}$ and $x\in \mathbb{B}$,
\begin{equation*}
D^t_s \int_{\mathbb{B}} R_b(x,y) f(y) d\nu_b(y) = \int_{\mathbb{B}} D^t_s R_b(x,y) f(y) d\nu_b(y).
\end{equation*}
\end{lemma}

In some cases, $D^t_s$ can be written as an integral operator. More precisely we have the following result of \cite[Corollary 2.5]{DU1}.
\begin{corollary}\label{Corollary-Dst-Integral}
Let $s>-1$ and $f\in L^1_s \cap h(\mathbb{B})$. For every $t\in \mathbb{R}$,
\begin{equation}\label{Dst - Integral operator}
  D^t_s f(x) = \int_{\mathbb{B}} R_{s+t}(x,y) f(y) d\nu_s(y).
\end{equation}
\end{corollary}

The following lemma states that when $f\in b^{1}_{b}$$(b>-1)$, the operator $T_{bc}$ acts like $D^t_s$.
\begin{lemma}\label{Lemma-T act as Dst}
Let $b>-1$, $c\in \mathbb{R}$ and $f\in b^1_b $. Then
\begin{equation*}
  \frac{1}{V_{b}}T_{bc}f(x) = \int_{\mathbb{B}} R_{c}(x,y) f(y) d\nu_b(y)=D^{c-b}_b f(x).
\end{equation*}
\end{lemma}
\begin{proof}
It is obvious from the definition of   $d\nu_b$ and the previous corollary.
\end{proof}

The following result is significant in our necessity proofs.

\begin{lemma}\label{Composition of T, Dst}
If $b+t>-1$, then $T_{bc}I_{b}^{t}h=C D_{b}^{c-b}h$ for $h\in b^1_b $. As consequences, $D_{c}^{b-c}T_{bc}I_{b}^{t}h=Ch$ for $h\in b^1_b $ and
$T_{bc}I_{b}^{t}D_{c}^{b-c}h=Ch$ for $h\in b^1_c $.
\end{lemma}
\begin{proof}
If $b+t>-1$ and $h\in b^1_b $, then $D_{b}^{t}h\in  b^1_{b+t}\subset L^1_{b+t}$ by Lemma \ref{Apply-Dst} (i). Since
\begin{equation*}
 T_{bc}I_{b}^{t}h(x) = \int_{\mathbb{B}} R_{c}(x,y) D_{b}^{t}h(y)(1-|y|^2)^{b+t} d\nu(y),
\end{equation*}
we have
\begin{equation*}
 T_{bc}I_{b}^{t}h =T_{b+t,c}D_{b}^{t}h=C D_{b+t}^{c-b-t} D_{b}^{t}h=CD_{b}^{c-b}h.
\end{equation*}
by Lemma \ref{Lemma-T act as Dst} and (\ref{Additive-Dst}). The identities on  triple compositions are just consequences of the identities in (\ref{inverse of Dst}).
\end{proof}

We require the inclusion relations between harmonic Bergman-Besov and we\-igh\-ted Bloch spaces in necessity proofs. We refer to \cite{DU2} for results on inclusions where also references  to earlier work can be found.

First, we single out  the following simple inclusions:
\begin{equation}\label{Inclusion}
 b^{p}_\alpha \subset b^{p}_{\beta} \qquad  and \qquad  b^{\infty}_\alpha \subset b^{\infty}_{\beta} \qquad ( \alpha \leq \beta).
\end{equation}

We have the following inclusion relations between harmonic Bergman-Besov spaces. For  proofs see \cite[Theorems 1.1 and 1.2]{DU2}.
\begin{theorem}\label{Besov-Besov1}
Let $0<q<p<\infty$ and $\alpha,\beta\in \mathbb{R}$. Then
\[
b^p_\alpha \subset b^q_\beta \quad \text{if and only if} \quad \frac{\alpha+1}{p} < \frac{\beta+1}{q}.
\]
\end{theorem}

\begin{theorem}\label{Besov-Besov2}
Let $0<p\leq q<\infty$ and $\alpha,\beta\in\mathbb{R}$. Then
\[
b^p_\alpha \subset b^q_\beta \quad \text{if and only if} \quad \frac{\alpha+n}{p} \leq \frac{\beta+n}{q}.
\]
\end{theorem}

We also have the following inclusion relation between a Bergman-Besov space $b^p_\alpha$ and a weighted Bloch space $b^\infty_\beta$. For a proof see \cite[Theorem 1.3]{DU2}.
\begin{theorem}\label{Besov-Bloch}
Let $0<p<\infty$ and $\alpha,\beta\in \mathbb{R}$. Then

\begin{enumerate}
  \item[(i)] $b^\infty_\beta \subset b^p_\alpha$ if and only if $\displaystyle \beta<\frac{\alpha+1}{p}$.
  \item[(ii)] $b^p_\alpha \subset b^\infty_\beta$ if and only if $\displaystyle \beta \geq \frac{\alpha+n}{p}$.
\end{enumerate}
\end{theorem}
Note that all the inclusions above are continuous, strict, and the best possible.

 We now mention two more theorems about inclusion relations that we will invoke later. The following theorem gives the inclusion relation between
$h^{\infty}$ and $b^{\infty}_{\alpha }$.

\begin{theorem} \label{hinfinity1}
Let  $\alpha \in \mathbb{R}$.
\begin{enumerate}
  \item[(i)] If $\alpha <0$, then $ b^{\infty}_\alpha\subset h^{\infty}$.
  \item[(ii)]If $\alpha \geq 0$, then $h^{\infty}\subset b^{\infty}_{\alpha }$.
\end{enumerate}
\end{theorem}

\begin{proof}
(i): Let $\alpha <0$. Pick $s,t\in \mathbb{R}$ such that $\alpha+t>0$ and $s>\alpha-1$ holds. Assume that  $f\in b^{\infty}_\alpha$. By (\ref{QsIst=f}) we have the following integral representation
\begin{equation*}
f(x)=\frac{V_s}{V_{s+t}}\int_{\mathbb{B}} R_s(x,y) I^{t}_{s} f(y) (1-|y|^{2})^{s} \, d\nu(y),
\end{equation*}
and therefore
\begin{equation*}
|f(x)|\lesssim \int_{\mathbb{B}} |R_s(x,y)| |I^{t}_{s} f(y)| (1-|y|^{2})^{s} \, d\nu(y).
\end{equation*}
Using that $\|f\|_{b^{\infty}_{\alpha }}=\|I^{t}_{s} f\|_{L^{\infty}_{\alpha}}=\sup_{x\in\mathbb{B}} (1-|x|^{2})^{\alpha}|I^{t}_{s} f(x)|$, we get
 $(1-|x|^{2})^{\alpha}|I^{t}_{s} f(x)| \-\leq \|f\|_{b^{\infty}_{\alpha }}$ for all  $x\in\mathbb{B}$. Thus
\begin{equation*}
 |f(x)|\lesssim \|f\|_{b^{\infty}_{\alpha}}\int_{\mathbb{B}} |R_s(x,y)| (1-|y|^{2})^{s-\alpha} \, d\nu(y).
\end{equation*}
Since $s-\alpha>-1$ and $n+s-(n+s-\alpha)=\alpha<0$, by Lemma \ref{norm-kernel} we have
 $|f(x)|\lesssim \|f\|_{b^{\infty}_{\alpha }}$ for all $x\in\mathbb{B}$. We conclude that $f\in h^{\infty}$.

(ii) Let $f\in h^{\infty}$. First, we take $\alpha>0$. So that it is enough to show that $f\in \mathcal{L}^{\infty}_{\alpha }$. Since $f\in h^{\infty}$, there exist an $M>0$ such that $|f(x)|\leq M$ for all  $x\in\mathbb{B}$. Together with  $(1-|x|^{2})^{\alpha}\leq 1$, this yields $(1-|x|^{2})^{\alpha}|f(x)|\leq M$  for all $x\in\mathbb{B}$. Hence we have $f\in \mathcal{L}^{\infty}_{\alpha }$ and this implies $f\in b^{\infty}_{\alpha }$.

Let now $\alpha=0$. This time we must show that $\sup_{x\in\mathbb{B}} (1-|x|^{2})|\nabla f(x)|<\infty$.  Since $f\in h^{\infty}$, again there exist an $M>0$ such that $|f(x)|\leq M$ for all  $x\in\mathbb{B}$. By Cauchy's estimate (see \cite[2.4]{ABR}), there exists a positive constant $C$ such that
\begin{equation*}
|\nabla u(x)|\leq \frac{C M}{r},
\end{equation*}
for every $x\in \mathbb{B}$. Since $1+|x|\leq 2$ when $x\in\mathbb{B}$, we obtain
\begin{equation*}
 (1-|x|^{2})|\nabla u(x)|\leq \frac{2 C  M (1-|x|)}{r}\leq 4 C M.
\end{equation*}
 Hence $f\in b^{\infty}_{0 }$.
\end{proof}

We also have the following inclusion theorem between $h^{\infty}$ and $b^{p}_{\alpha }$. For a proof see Section 13.1 of \cite{GKU2} and discussion in there  when $1\leq p< \infty$ and \cite[Theorem 5.1]{DOG} when $0< p< 1$.
\begin{theorem}\label{hinfinity2}
Let $0<p<\infty$ and $\alpha\in \mathbb{R}$. Then
\[
b^p_\alpha \subset h^\infty \quad \text{if and only if} \quad \alpha < -n, \, \text{ or } \, \alpha=-n \quad \text{and} \quad 0<p\leq 1.
\]
\end{theorem}

\section{Proofs of Necessity Parts of Theorems \ref{Theorem-Boundedness of T,S 1.1}--\ref{Theorem-Boundedness of T,S 3.4}}\label{Section-Necessity Proofs}
In this section, we obtain necessary conditions for the boundedness of the o\-pe\-ra\-tor $T_{bc}$, that is, (i) implies (iii) in all of our seven theorems.
Before the necessity proofs, we first show the following lemma and corollary after that.
\begin{lemma}\label{Lemma-Kernel-comparasion}
Let $a,a_{1}, a_{2} \in\mathbb{R}$. Define $ \mathcal{R}_{a}(x,y)$ such that
\begin{equation*}
 \mathcal{R}_{a}(x,y)
:=\begin{cases}
1,&\text{if $\, a<-n$};\\
1+\log \dfrac{1}{[x,y]},&\text{if $\, a=-n$};\\
\dfrac{1}{[x,y]^{n+a}},&\text{if $\, a>-n$}.
\end{cases}
\end{equation*}
If $a_{1}< a_{2}$, then we have $ \mathcal{R}_{a_{1}}(x,y)\lesssim \mathcal{R}_{a_{2}}(x,y)$ for all $x,y \in \mathbb{B}$.
\end{lemma}

\begin{proof}
There are five cases $a_{1}<a_{2}<-n$, $-n<a_{1}<a_{2}$, $a_{1}<-n<a_{2}$, $a_{1}<a_{2}=-n$ and $-n=a_{1}<a_{2}$ all of them can be elementary verified. If $a_{1}<a_{2}<-n$, it is clear that this estimate holds. If $-n<a_{1}<a_{2}$, we write
\begin{equation*}
\frac{1}{[x,y]^{n+a_{1}}}=\frac{[x,y]^{a_{2}-a_{1}}}{[x,y]^{n+a_{2}}}\lesssim \frac{1}{[x,y]^{n+a_{2}}}.
\end{equation*}
Now let $a_{1}<n<a_{2}$. Since $[x,y]\leq 2$ and $0<n+a_{2}$, its obvious that $\frac{1}{[x,y]}\geq \frac{1}{2}$ and thus we have
\begin{equation*}
1 \lesssim \frac{1}{[x,y]^{n+a_{2}}}.
\end{equation*}
If $a_{1}<a_{2}=-n$, then $1+\log\frac{1}{[x,y]}\geq 1+\log\frac{1}{2}\gtrsim 1$. Finally, let $-n=a_{1}<a_{2}$. Note that, $1+\log\frac{1}{[x,y]}$ and $\frac{1}{[x,y]^{n+a_{2}}}$ are bounded both above and below when $[x,y]$ away from zero. On the other side, $1+\log\frac{1}{[x,y]}$ is dominated by $\frac{1}{[x,y]^{n+a_{2}}}$ when $[x,y]$ near zero, because
\begin{equation*}
\lim_{t \to 0}\frac{\log\frac{1}{t}}{(\frac{1}{t})^{\delta}}=0,
\end{equation*}
for $\delta=n+a_{2}>0$ and $t=[x,y]$. Hence, if $a_{1}< a_{2}$ then $ \mathcal{R}_{a_{1}}(x,y)\lesssim \mathcal{R}_{a_{2}}(x,y)$ for all $x,y \in \mathbb{B}$.
\end{proof}

\begin{corollary}\label{Corollary-Lemma-Kernel-comparasion}
If $S_{bd}:L^{p}_{\alpha}\to L^{q}_{\beta}$ is bounded and $c<d$, then $S_{bc}:L^{p}_{\alpha}\to L^{q}_{\beta}$ is also bounded
\end{corollary}
\begin{proof}
This is just because $|R_{c}(x,y)|\lesssim |R_{d}(x,y)|$ by the previous lemma.
\end{proof}

Firstly, we derive the first inequality in (iii) of each theorem. In this section, we do not need to assume $\beta>-1$ when $q<\infty$ or $\beta\geq 0$ when $q=\infty$ since the boundedness of $T_{bc}$ implies one of them by Corollary \ref{Remark-T,S with beta bigger than -1}. More precisely, the following theorem gives the first necessary conditon for all of Theorems \ref{Theorem-Boundedness of T,S 1.1}--\ref{Theorem-Boundedness of T,S 3.4}.

\begin{theorem}\label{the first necessary condition }
Let $b, c, \alpha, \beta \in \mathbb{R}$ and $1\leq p, q\leq\infty$. Suppose that $T_{bc}$ is bounded from $L^p_\alpha$ to $L^q_\beta$, then
$\alpha+1\leq p(b+1)$ for $1=p\leq q\leq\infty$, also $\alpha-1<b$ for $1\leq q <p=\infty$ and the strict inequality $\alpha+1<p(b+1)$ holds for the remainig cases.
\end{theorem}

\begin{proof}
The proof can be seperated in three cases depending on the value of $p$. We first show the case $1<p<\infty$. Consider $f_{uv}$ with $u=-(1+\alpha)/p$ and $v=1$ so that $f_{uv}\in L^p_\alpha$ by Lemma \ref{when fuv in Lpq}. Then its clear that $T_{bc}f_{uv}\in L^q_\beta$ and Lemma \ref{when Tfuv is finite} implies $b+u>-1$. This yields $(1+\alpha)/p<1+b$ with the value of $u$ chosen.

We now show the second case $p=1$. Consider $f_{uv}$ with $u>-(1+\alpha)$ and $v=0$ so that $f_{u0}\in L^1_\alpha$ by Lemma \ref{when fuv in Lpq}. Then $T_{bc}f_{u0}\in L^q_\beta$ and Lemma \ref{when Tfuv is finite} implies $b+u>-1$. Writing $u=-1-\alpha+\varepsilon$ with $\varepsilon>0$, we obtain $\alpha<b+\varepsilon$. This is just $\alpha\leq b$.

The last case is $p=\infty$. Let now  $f_{uv}$ with $u=-\alpha$  so that $f_{u0}\in \mathcal{L}^\infty_\alpha$ by Lemma \ref{when fuv in Lpq}. Then $T_{bc}f_{u0}\in L^q_\beta$ and Lemma \ref{when Tfuv is finite} implies $b-\alpha>-1$.
\end{proof}

We next derive the second inequality in (iii) of each theorem. As indicated before, we do this by an original method depends on the inclusion relations between Bergman-Besov and weighted Bloch spaces appears in \cite{KU1}. A key step of this method is  Lemma \ref{Composition of T, Dst}.

\begin{theorem}\label{the second necessary condition }
Let $b, c, \alpha, \beta \in \mathbb{R}$ and $1\leq p, q\leq\infty$. Suppose that $T_{bc}$ is bounded from $L^p_\alpha$ to $L^q_\beta$. Then
 $c\leq b+\frac{n+\beta}{q}-\frac{n+\alpha}{p}$ for $1\leq p\leq q<\infty$, $c< b+\frac{1+\beta}{q}-\frac{1+\alpha}{p}$ for $1\leq q<p<\infty$, $c\leq b+\beta-\frac{n+\alpha}{p}$ for $1\leq p<q=\infty$ and the strict inequality holds for $p\neq1$ when $\beta=0$, also  $c< b+\frac{1+\beta}{q}-\alpha$  for $1 \leq q< p=\infty$, $c\leq b+\beta-\alpha$ for $p=q=\infty$ and the strict inequality holds  when $\beta=0$.
\end{theorem}

\begin{proof}
First note that, the boundedness of $T_{bc}$ implies $\beta>-1$ when $q<\infty$ and $\beta\geq 0$ when $q=\infty$ by Corollary \ref{Remark-T,S with beta bigger than -1}.

Now we handle the all cases in four groups. The first group again consist of the cases $1\leq p\leq q<\infty$ and $1\leq q< p<\infty$.
Let $h \in b^p_\alpha$. In order to able to use Lemma \ref{Composition of T, Dst}, we need to show that $h\in b^1_b$. In the cases $1< p\leq q<\infty$ and $1\leq q<p<\infty$, we have $(1+\alpha)/p<1+b$ by the first necessary condition, and then Theorem \ref{Besov-Besov1} gives $h\in b^1_b$. In the case $1=p\leq q<\infty$, we have $\alpha\leq b$ again by the first necessary condition, and then (\ref{Inclusion}) shows $h\in b^1_b$. Pick $t$ such that $\alpha+pt>-1$. Together with the first necessary condition, it is easy to check that $b+t>-1$. We will consider the composition of the bounded maps
 \begin{equation*}
 b^p_\alpha \stackrel{\mathrm{I_{b}^{t}}}{\xrightarrow{\hspace*{1cm}}} L^p_\alpha \stackrel{\mathrm{T_{bc}}}{\xrightarrow{\hspace*{1cm}}}  b^q_\beta
 \stackrel{\mathrm{D_{c}^{b-c}}}{\xrightarrow{\hspace*{1cm}}} b^q_{\beta+q(b-c)}.
\end{equation*}
Note that since $T_{bc}h$ is harmonic and $\beta>-1$ range of $T_{bc}$ is $b^q_\beta \subset L^q_\beta $. Lemma \ref{Composition of T, Dst} yields that $b^p_\alpha$ is imbedded in $b^q_{\beta+q(b-c)}$ by the inclusion map. But by Theorem \ref{Besov-Besov2} this is possible only if $(\alpha+n)/p\leq (\beta+q(b-c)+n)/q$ which is equivalent to $c\leq b+\frac{n+\beta}{q}-\frac{n+\alpha}{p}$ in the case $1\leq p\leq q<\infty$. Similarly, by Theorem \ref{Besov-Besov1} this is possible only if $(\alpha+1)/p\leq (\beta+q(b-c)+1)/q$ that is  $c< b+\frac{1+\beta}{q}-\frac{1+\alpha}{p}$ in the case $1\leq q< p<\infty$.

 The second group  consist of the case $1\leq p< q=\infty$.
Let $ H\in b^p_{\alpha+p(c-b)}$. By Lemma \ref{Apply-Dst} (i), $D_{c}^{b-c}H=h\in b^p_\alpha$. Exactly as in the proof of the first group of cases,    $h\in b^1_b$. Pick $t$ such that $\alpha+pt>-1$ and this gives $b+t>-1$ with the first necessary condition.  We will consider the composition of the bounded maps
 \begin{equation*}
b^p_{\alpha+p(c-b)} \stackrel{\mathrm{D_{c}^{b-c}}}{\xrightarrow{\hspace*{1cm}}} b^p_\alpha \stackrel{\mathrm{I_{b}^{t}}}{\xrightarrow{\hspace*{1cm}}}  L^p_\alpha
 \stackrel{\mathrm{T_{bc}}}{\xrightarrow{\hspace*{1cm}}} b^\infty_{\beta}.
\end{equation*}
Note that since $T_{bc}h$ is harmonic, range of $T_{bc}$ is $b^\infty_\beta \subset \mathcal{L}^\infty_\beta $ when $\beta>0$ and $\mathcal{L}^\infty\cap h(\mathbb{B})=h^{\infty} \subset L^\infty $ when $\beta=0$. Lemma \ref{Composition of T, Dst} yields that $b^p_{\alpha+p(c-b)}$ is imbedded in $b^\infty_{\beta}$ when $\beta>0$ and in $h^\infty$ when $\beta=0$ by the inclusion map. By Theorem \ref{Besov-Bloch} (ii) this is possible only if  $c\leq b+\beta-\frac{n+\alpha}{p}$ when $\beta>0$. Similarly, by Theorem \ref{hinfinity2}  this is possible only if  $c< b-\frac{n+\alpha}{p}$ in the case $1< p< q=\infty$ and it is possible only if  $c\leq b-(n+\alpha)$ in the case $1= p, q=\infty$   when $\beta=0$.

 The third group  consist of the case $1\leq q< p=\infty$.
Let $ h\in b^\infty_{\alpha}$. The first necessary condition gives $\alpha-1<b$ and then Theorem \ref{Besov-Bloch}(i) yields  $h\in b^1_b$. Pick $t$ such that $\alpha+t>0$ and this gives $b+t>-1$ with the first necessary condition.  We will consider the composition of the bounded maps
 \begin{equation*}
b^\infty_{\alpha} \stackrel{\mathrm{I_{b}^{t}}}{\xrightarrow{\hspace*{1cm}}} \mathcal{L}^\infty_\alpha \stackrel{\mathrm{T_{bc}}}{\xrightarrow{\hspace*{1cm}}}  b^q_\beta
 \stackrel{\mathrm{D_{c}^{b-c}}}{\xrightarrow{\hspace*{1cm}}} b^q_{\beta+q(b-c)}.
\end{equation*}
Note that since $T_{bc}h$ is harmonic and $\beta>-1$, range of $T_{bc}$ is $ b^q_\beta\subset  L^q_{\beta}$.  Lemma \ref{Composition of T, Dst} yields that $b^\infty_{\alpha}$ is imbedded in $b^q_{\beta+q(b-c)}$  by the inclusion map. By Theorem \ref{Besov-Bloch} (i) this is possible only if  $c< b+\frac{1+\beta}{q}-\alpha$.

The last group  consist of the case $ p= q=\infty$.
Let $ H\in b^\infty_{\alpha+(c-b)}$. By Lemma \ref{Apply-Dst} (ii), $D_{c}^{b-c}H=h\in b^\infty_\alpha$. Exactly as in the proof of the third group of cases, $h\in b^1_b$. Pick $t$ such that $\alpha+t>0$ and this gives $b+t>-1$ with the first necessary condition.  We will consider the composition of the bounded maps
 \begin{equation*}
b^\infty_{\alpha+(c-b)} \stackrel{\mathrm{D_{c}^{b-c}}}{\xrightarrow{\hspace*{1cm}}} b^\infty_\alpha \stackrel{\mathrm{I_{b}^{t}}}{\xrightarrow{\hspace*{1cm}}}  \mathcal{L}^\infty_\alpha
 \stackrel{\mathrm{T_{bc}}}{\xrightarrow{\hspace*{1cm}}} b^\infty_{\beta}.
\end{equation*}
Note that since $T_{bc}h$ is harmonic, range of $T_{bc}$ is $b^\infty_\beta \subset \mathcal{L}^\infty_\beta $ when $\beta>0$ and $\mathcal{L}^\infty\cap h(\mathbb{B})=h^{\infty} \subset L^\infty $ when $\beta=0$. Lemma \ref{Composition of T, Dst} yields that $b^\infty_{\alpha+(c-b)}$ is imbedded in $b^\infty_{\beta}$ when $\beta>0$ and in $h^\infty$ when $\beta=0$ by the inclusion map. By (\ref{Inclusion}) this is possible only if  $c\leq b+\beta-\alpha$ when $\beta>0$. Similarly, by Theorem \ref{hinfinity1} this is possible only if  $c< b-\alpha$ in when $\beta=0$.
\end{proof}

Finaly, we must prove that in the case $1=p\leq q \leq \infty$, if one of the inequalities in (iii) of Theorems \ref{Theorem-Boundedness of T,S 1.2} and \ref{Theorem-Boundedness of T,S 3.2} is an equality, then the other must be a strict inequality. Our method of proof will be an adaptation of the reasoning used in Theorem 6.3 of \cite{KU1}.

\begin{theorem}\label{the strict inequality condition}
Let $b, c, \alpha, \beta \in \mathbb{R}$ and $1\leq p, q\leq\infty$. Suppose that $T_{bc}$ is bounded from $L^p_\alpha$ to $L^q_\beta$. Then
equality cannot hold simultaneously in the inequalities  of Theorems \ref{Theorem-Boundedness of T,S 1.2} and \ref{Theorem-Boundedness of T,S 3.2}.
\end{theorem}
\begin{proof}
First note that, the boundedness of $T_{bc}$ implies $\beta>-1$ when $\infty<q$ and $\beta\geq 0$ when $q=\infty$ by Corollary \ref{Remark-T,S with beta bigger than -1}.

If $\alpha=b$ and $c=b+\beta-\alpha$ simultaneously in the case $1=p=q$, then also $c=\beta>-1$ and $T^{*}_{bc}=T_{\beta\beta}:L^{\infty}\to L^{\infty}$ is bounded. Let
\begin{equation*}
f_{x}(y)=\begin{cases}\frac{|R_{\beta}(x,y)|}{R_{\beta}(x,y)},&\text{if $R_{\beta}(x,y)\neq 0$};\\
 1,&\text{if $\, R_{\beta}(x,y)=0$},
\end{cases}
\end{equation*}
which is a uniformly bounded family for $x\in \mathbb{B}$. The same is true also of $\{T_{\beta\beta}f_{x}\}$. But
\begin{equation*}
T_{\beta\beta}f_{x}(x)=\int_{\mathbb{B}} |R_{\beta}(x,y)|  (1-|y|^2)^{\beta} d\nu(y)\sim
1+\log\dfrac1{1-|x|^2}
\end{equation*}
by  Lemma \ref{norm-kernel}, this contradicts to the uniform boundedness.

If $\alpha=b$ and $c=b+(n+\beta)/q-(n+\alpha)$ simultaneously in the case $1=p<q<\infty$, then also $c=(n+\beta)/q-n$ and $T^{*}_{bc}=T_{\beta c}:L^{q'}_{\beta}\to L^{\infty}$ is bounded with $q'>1$.  By Theorem \ref{hinfinity2}, there is an unbounded $g\in b^{q'}_{-n}$. Then $h=D^{(n+\beta)/q'}_{\beta-(n+\beta)/q'}g \in b^{q'}_{\beta} \subset L^{q'}_{\beta} \subset L^{1}_{\beta}$. By Lemma \ref{Lemma-T act as Dst} and (\ref{Additive-Dst}), we obtain
\begin{equation*}
T^{*}_{bc}h=T_{\beta c}h=V_{\beta}D^{c-\beta}_{\beta}h=V_{\beta}D^{c-\beta}_{\beta}D^{(n+\beta)/q'}_{\beta-(n+\beta)/q'}g=V_{\beta}g.
\end{equation*}
Nevertheless $g \notin L^{\infty}$, and this contradicts that $T^{*}_{bc}:L^{q'}_{\beta}\to L^{\infty}$.

If $\alpha=b$ and $c=b+\beta-(\alpha+n)$ simultaneously in the case $1=p, q=\infty$, then also $c=\beta-n$. For $i=1,2,...$, let $x_{i}=(1-1/i,0,...,0)$ and $E_{i}$ the ball of radius $1/2i$ centered at $x_{i}$, and define
\begin{equation*}
f_{i}(x)=\frac{V_{\alpha}\, \chi_{E_{i}}(x)}{\nu(E_{i})(1-|x|^{2})^{\alpha}}.
\end{equation*}
Clearly $f_{i}\in L^{1}_{\alpha}$ and $\|f_{i}\|_{L^{1}_{\alpha}}=1$ for every $i$. Then $\{T_{bc}f_{i}\}=\{T_{\alpha,\beta-n}f_{i}\}$ is a uniformly bounded family. By the mean value property,
\begin{equation*}
T_{\beta,\beta-n}f_{i}(y)=\frac{V_{\alpha}}{\nu(E_{i})}\int_{E_{i}} R_{\beta-n}(x,y)d\nu(x)=V_{\alpha}R_{\beta-n}(y,x_{i}).
\end{equation*}
But
\begin{equation*}
T_{\beta,\beta-n}f_{i}(x_{i})=V_{\alpha}R_{\beta-n}(x_{i},x_{i})\sim
\begin{cases}
\dfrac{1}{(1-|x_{i}|^2)^{\beta}},&\text{if $\, \beta>0$};\\
1+\log \dfrac{1}{1-|x_{i}|^2},&\text{if $\, \beta=0$};
\end{cases}
\end{equation*}
by  Lemma \ref{x-equal-y}, this contradicts to the uniform boundedness.
\end{proof}

\section{Proofs of Sufficiency Parts of Theorems \ref{Theorem-Boundedness of T,S 1.1}--\ref{Theorem-Boundedness of T,S 3.4}}\label{Section-Sufficiency Proofs}
In this section we will present the proofs that the inequalities in (iii) of Theorems \ref{Theorem-Boundedness of T,S 1.1}--\ref{Theorem-Boundedness of T,S 3.4} imply the boundedness of $S_{bc}$. By Corollary \ref{Corollary-Lemma-Kernel-comparasion}, it is enough to prove this only for large values of $c$. In all theorems except Theorem \ref{Theorem-Boundedness of T,S 3.2}, there are values of $c>-n$ satisfying the inequalities in (iii), thus  we make this the standing assumption in this section. For Theorem \ref{Theorem-Boundedness of T,S 3.2}, we deal with  the values of $c$ separately.

 We consider each theorem separately since each of the cases has a sufficiently different proof those from the others. Throughout this section, we assume that the two inequalities in (iii) hold.

 The following sufficiency proof of Theorem \ref{Theorem-Boundedness of T,S 1.1} follows the same lines as the proofs of \cite[Lemma 6]{Z}.

\begin{proof}[Proof of sufficiency for Theorem \ref{Theorem-Boundedness of T,S 1.1}]\label{Sufficiency Proofs 1.1}
First, taking $c$ to have its largest value
\begin{equation}\label{largest value of c}
c=b+\frac{n+\beta}{q}-\frac{n+\alpha}{p}
\end{equation}
causes no loss of generality by Corollary \ref{Corollary-Lemma-Kernel-comparasion}. We have $c>-n$ by the condition $\alpha+1<p(b+1)$ and $\beta>-1$.
We employ the Schur test in Theorem \ref{Schur test 1} with the measures $\mu=\nu_{\alpha}, \upsilon=\nu_{\beta}$ and the kernel $K(x,y)=\frac{(1-|x|^2)^{b-\alpha}}{[x,y]^{n+c}}$ which together give us  us the desired boundedness of the operator $S_{bc}$. Thus we need to find two positive constant $\gamma$ and $\delta$ such that $\gamma+\delta=1$ and two strictly positive functions $\phi(x)=(1-|x|^2)^{A}$ and $\psi(y)=(1-|y|^2)^{B}$ on $\mathbb{B}$ with $A,B\in\mathbb{R}$ to be determined. The two inequalities that need to be satisfied for the Schur test are
\begin{align*}
&\int_{\mathbb{B}} \frac{(1-|x|^2)^{(b-\alpha)\gamma p'}}{[x,y]^{(n+c)\gamma p'}}(1-|x|^2)^{Ap'}(1-|x|^2)^{\alpha}d\nu(x)\lesssim (1-|y|^2)^{Bp'},\\
&\int_{\mathbb{B}} \frac{(1-|x|^2)^{(b-\alpha)\delta q}}{[x,y]^{(n+c)\delta q}}(1-|y|^2)^{Bq}(1-|y|^2)^{\beta}d\nu(y)\lesssim (1-|x|^2)^{Aq}.
\end{align*}
One way to satisfy them is by matching the growth rates of their two sides, that is, the powers of the $1-|\cdot|^{2}$. This is possible if $A,B<0$ and
\begin{equation}\label{possible from Lemma Integral-[x,y]}
\begin{split}
&-Bp'=(n+c)\gamma p'-n-(b-\alpha)\gamma p'-Ap'-\alpha,\\
&-Aq=(n+c)\delta q-n-Bq-\beta-(b-\alpha)\delta q,
\end{split}
\end{equation}
by Lemma \ref{Integral-[x,y]}. But we must also make sure that the conditions of Lemma \ref{Integral-[x,y]} for this to happen are met, that is,
\begin{equation}\label{conditions of Lemma Integral-[x,y]}
\begin{split}
F_{1}:=(b-\alpha)\gamma p'+Ap'+\alpha&>-1,\\
Bq+\beta&>-1,\\
F_{2}:=(n+c)\gamma p'-n-(b-\alpha)\gamma p'-Ap'-\alpha&>0,\\
(n+c)\delta q-n-Bq-\beta&>0.
\end{split}
\end{equation}
Keep in mind that the two equations  in (\ref{possible from Lemma Integral-[x,y]}) are linearly dependent. The variables $A,B$ and $\delta$ are determined in the following way. We first choose an $B$ to satisfy the second inequality in (\ref{conditions of Lemma Integral-[x,y]}); so
\begin{equation}\label{ value of B}
-\frac{1+\beta}{q}<B<0.
\end{equation}
This is possible since $\beta>-1$. Next we pick a $\delta$ to satisfy the fourth inequality in (\ref{conditions of Lemma Integral-[x,y]}) and naturally let $\gamma=1-\delta$; so we take
\begin{equation}\label{ value of gamma-delta}
\begin{split}
\delta&=\frac{1}{n+c}(B+\frac{n+\beta}{q}+\varepsilon),\\
\gamma&=\frac{1}{n+c}(-B+n+b-\frac{n+\alpha}{p}-\varepsilon)
\end{split}
\end{equation}
with $\varepsilon>0$ by (\ref{largest value of c}). Using the chosen values of $B,\gamma$ and $\delta$, we then solve for $A$ from, say, the second equation in (\ref{possible from Lemma Integral-[x,y]}), and simplify it using the definition of $\delta$; so
\begin{equation}\label{value of A}
\begin{split}
A&=(b-\alpha)\delta-(n+c)\delta+B+ \frac{n+\beta}{q}\\
&=(b-\alpha)\delta-B-\frac{n+\beta}{q}-\varepsilon+B+  \frac{n+\beta}{q}=(b-\alpha)\delta-\varepsilon.
\end{split}
\end{equation}
Finally, we must check that the remaining first and third inequalities in (\ref{conditions of Lemma Integral-[x,y]}) hold for some $\varepsilon>0$. Substituting in the value of $A$ from (\ref{value of A}), since $\gamma+\delta=1$,
\begin{align*}
F_{1}+1&=(b-\alpha)\gamma p'+(b-\alpha)\delta p'-\varepsilon p'+\alpha+1=(b-\alpha) p'+\alpha+1-\varepsilon p'\\
&=p'(b-\alpha+\frac{\alpha+1}{p'})-\varepsilon p'=p'(b+1-\frac{\alpha+1}{p})-\varepsilon p'\\
&=p'(b+1+\frac{\alpha+1}{p}-\varepsilon)>0
\end{align*}
by the condition $\alpha+1<p(b+1)$ provided $\varepsilon<b+1-\frac{\alpha+1}{p}$. Substituting in for $A$ and $\gamma$  from (\ref{ value of gamma-delta}) and (\ref{value of A}), again since $\gamma+\delta=1$,
\begin{align*}
F_{2}&=p'\left((n+c)\gamma-(b-\alpha)\gamma-(b-\alpha)\delta+\varepsilon-\frac{n+\alpha}{p'}\right)\\
&=p'\left(-B+n+b-\frac{n+\alpha}{p}-\varepsilon-b+\alpha+\varepsilon-\frac{n+\alpha}{p'}\right)\\
&=-p'B>0
\end{align*}
by (\ref{ value of B}). Hence for
\begin{equation*}
0<\varepsilon<b+1-\frac{\alpha+1}{p},
\end{equation*}
Theorem \ref{Schur test 1} using the chosen functions $\phi$ and $\psi$ with the powers in (\ref{value of A}) and (\ref{ value of B}) applies with the parameters in  (\ref{ value of gamma-delta}) and proves that $S_{bc}$ is bounded from $L^{p}_{\alpha}$ to $L^{q}_{\beta}$ with $1<p\leq q<\infty$ when the inequalities in (iii) hold.
\end{proof}

\begin{proof}[Proof of sufficiency for Theorem \ref{Theorem-Boundedness of T,S 1.2}]\label{Sufficiency Proofs 1.2}
First, let $1=p=q$ and $f\in L_{\alpha}^{1}$. Writing the $L_{\beta}^{1}$ norm of $S_{bc}f$ explicitly and applying Fubini theorem, we get that
\begin{align*}
\|S_{bc}f\|_{L_{\beta}^{1}}&\lesssim \int_{\mathbb{B}} \int_{\mathbb{B}} |R_{c}(x,y)||f(x)|(1-|x|^2)^{b}d\nu(x)(1-|y|^2)^{\beta}d\nu(y)\\
&\lesssim\int_{\mathbb{B}}|f(x)|(1-|x|^2)^{b}\int_{\mathbb{B}} \frac{(1-|y|^2)^{\beta}}{[x,y]^{n+c}}d\nu(y)d\nu(x)\\
&=V_{\alpha}\int_{\mathbb{B}}|f(x)|(1-|x|^2)^{b-\alpha}\int_{\mathbb{B}} \frac{(1-|x|^2)^{\beta}}{[x,y]^{n+c}}d\nu(y)d\nu_{\alpha}(x).
\end{align*}
Let $J(x)$ be that part of the integrand of the outer integral multiplying $|f(x)|$. We will show that $J$ is bounded on $\mathbb{B}$ by using Lemma \ref{Integral-[x,y]} since $\beta>-1$ as required.

Consider $(n+c)-n-\beta=c-\beta$. Firstly, if $c-\beta<0$, then the integral in $J(x)$ is bounded and $J(x)$ is also bounded since $b-\alpha\geq 0$ by the first inequality of (iii). Next, if $c-\beta=0$, then the integral in $J(x)$ is $1+\log (1-|x|^{2})^{-1}$. But this time $b>\alpha$  since the two inequalities are the same in (iii) and $b=\alpha$  can not hold as stated in Theorem \ref{the strict inequality condition}. Then $J(x)$ is bounded by (\ref{obvious-inequalities}). Lastly, if $c-\beta>0$, then $J(x)\sim (1-|x|^2)^{b-\alpha-c+\beta}$. But by the second inequality in (iii) we have  $b-\alpha-c+\beta\geq 0$ and thus $J(x)$ is bounded once again. Therefore $\|S_{bc}f\|_{L_{\beta}^{1}}\lesssim \|f\|_{L_{\alpha}^{1}}$ and $S_{bc}$ is bounded from $L^{1}_{\alpha}$ to $L^{1}_{\beta}$.

Next, let $1=p<q<\infty$ and $f\in b_{\alpha}^{1}$. Writing the $L_{\beta}^{q}$ norm of $S_{bc}f$ explicitly and using Lemma \ref{Minkowski int-inequality} with the measures $\nu_{\alpha}$ and $\nu_{\beta}$, we obtain
\begin{align*}
\|Sf\|_{L_{\beta}^{q}}&=\left(\int_{\mathbb{B}} \left|V_{\alpha}\int_{\mathbb{B}} R_{c}(x,y)f(x)(1-|x|^2)^{b-\alpha}d\nu_{\alpha}(x)\right|^{q}d\nu_{\beta}(y)\right)^{1/q}\\
&\lesssim \int_{\mathbb{B}} \left(\int_{\mathbb{B}} |R_{c}(x,y)|^{q}|f(x)|^{q}(1-|x|^2)^{(b-\alpha)q}d\nu_{\beta}(y)\right)^{1/q}d\nu_{\alpha}(x)\\
&\lesssim\int_{\mathbb{B}}|f(x)|(1-|x|^2)^{b-\alpha} \left(\int_{\mathbb{B}} \frac{(1-|y|^2)^{\beta}}{[x,y]^{(n+c)q}}d\nu(y)\right)^{1/q}d\nu_{\alpha}(x).
\end{align*}
Let $J(x)$ be that part of the integrand of the outer integral multiplying $|f(x)|$ for $x\in \mathbb{B}$. We will show that $J$ is bounded on $\mathbb{B}$ by using Lemma \ref{Integral-[x,y]} since $\beta>-1$ as required.

Now we consider $\rho=(n+c)q-n-\beta$. Firstly, if $\rho<0$, then the integral in $J(x)$ is bounded and $J(x)$ is also bounded since $b-\alpha\geq 0$ by the first inequality of (iii). Next, if $\rho=0$, then the integral in $J(x)$ is $1+\log (1-|x|^{2})^{-1}$. Then since the two inequalities are the same in (iii), we have $b>\alpha$ again by Theorem \ref{the strict inequality condition}. Therefore $J(x)$ is bounded by (\ref{obvious-inequalities}). Lastly, if $\rho>0$, then $J(x)\sim (1-|x|^2)^{b-\alpha-\rho/q}$. But by the second inequality in (iii), we have  $b-\alpha-\rho/q=b-\alpha-(n+c)+(n+\beta)/q\geq 0$ and thus $J(x)$ is bounded once again. Hence $\|Sf\|_{L_{\beta}^{q}}\lesssim \|f\|_{L_{\alpha}^{1}}$ and $S_{bc}$ is bounded from $L^{1}_{\alpha}$ to $L^{q}_{\beta}$ with  $1=p<q<\infty$.
\end{proof}

\begin{proof}[Proof of sufficiency for Theorem \ref{Theorem-Boundedness of T,S 2}]\label{Sufficiency Proofs 2}
First, let $1<q<p<\infty$. The proof of this case starts out as in the proof of sufficiency for Theorem \ref{Theorem-Boundedness of T,S 1.1}.
Now we employ the Schur test in Theorem \ref{Schur test 2} but with the same test data as sufficiency proof in the case $1<p<q<\infty$. So we have the $\mu=\nu_{\alpha}, \upsilon=\nu_{\beta}$ and the kernel $K(x,y)=\frac{(1-|x|^2)^{b-\alpha}}{[x,y]^{n+c}}$ which together give us $V_{\alpha}G=S_{bc}$. Thus we need to find  two strictly positive functions $\phi(x)=(1-|x|^2)^{A}$ and $\psi(y)=(1-|y|^2)^{B}$ on $\mathbb{B}$ with $A,B\in\mathbb{R}$ to be determined. Two of the three inequalities that need to be satisfied for the Schur test are
\begin{align*}
&\int_{\mathbb{B}} \frac{(1-|x|^2)^{b-\alpha}}{[x,y]^{n+c}}(1-|x|^2)^{Ap'}(1-|x|^2)^{\alpha}d\nu(x)\lesssim (1-|y|^2)^{Bq'},\\
&\int_{\mathbb{B}} \frac{(1-|x|^2)^{b-\alpha}}{[x,y]^{n+c}}(1-|y|^2)^{Bq}(1-|y|^2)^{\beta}d\nu(y)\lesssim (1-|x|^2)^{Ap}.
\end{align*}
One way to satisfy them is by matching the growth rates of their two sides, that is, the powers of the $1-|\cdot|^{2}$. This is possible if $A,B<0$ and
\begin{equation}\label{possible from Lemma Integral-[x,y] 2}
\begin{split}
&-Bq'=c-(b+Ap'),\\
&-Ap=c-(Bq+\beta)-(b-\alpha),
\end{split}
\end{equation}
by Lemma \ref{Integral-[x,y]}. But we must also make sure that the conditions of Lemma \ref{Integral-[x,y]} for this to happen are met, that is,
\begin{equation}\label{conditions of Lemma Integral-[x,y] 2}
\begin{split}
b+Ap'&>-1, \quad Bq+\beta>-1,\\
c-(b+Ap')&>0, \quad c-(Bq+\beta)>0.
\end{split}
\end{equation}
Substituting for $p',q'$ in terms of $p,q$, we can write (\ref{possible from Lemma Integral-[x,y] 2}) as a system of two linear equations in the two unknowns $A,B$  as
\begin{equation}\label{linear equations system for A,B}
\begin{split}
p(q-1)A-q(p-1)B&=(c-b)(p-1)(q-1),\\
-pA+qB&=c-b+\alpha-\beta.
\end{split}
\end{equation}
this system has the following unique solution
\begin{equation}\label{solution-linear equations system for A,B}
\begin{split}
A&=\frac{(p-1)(q(c-b)+\alpha-\beta)}{p(q-p)},\\
B&=\frac{(q-1)(p(c-b)+\alpha-\beta)}{q(q-p)}
\end{split}
\end{equation}
for $A,B$. The second inequality in (iii) can be written in the form
\begin{equation}\label{ value of c}
c=b+\frac{1+\beta}{q}-\frac{1+\alpha}{p}-\varepsilon
\end{equation}
with $\varepsilon>0$. So by Corollary \ref{Corollary-Lemma-Kernel-comparasion}, it suffices to prove that $S_{bc}$ is bounded when (\ref{ value of c}) holds for small enough $\varepsilon>0$. Substituting this value of $c$ into (\ref{solution-linear equations system for A,B}), the solution
\begin{equation}\label{new form solution-les for A,B}
\begin{split}
A&=\frac{(p-1)}{p}\left(-\frac{1+\alpha}{p}+\frac{\varepsilon q}{p-q}\right),\\
B&=\frac{(q-1)}{q}\left(-\frac{1+\beta}{q}+\frac{\varepsilon p}{p-q}\right),
\end{split}
\end{equation}
Now, It remains to show that this solution satisfies all the necessary conditions for sufficiently small $\varepsilon>0$. Recall that $\beta>-1$.
First, by the  inequality $\alpha+1<p(b+1)$,
\begin{equation*}
c=b+\frac{1+\beta}{q}-\frac{1+\alpha}{p}-\varepsilon>\frac{1+\beta}{q}-1-\varepsilon>-(1+\varepsilon)>-n
\end{equation*}
provided $\varepsilon<n-1$. Next we need to check that the inequalities in (\ref{conditions of Lemma Integral-[x,y] 2}). By (\ref{new form solution-les for A,B}) and  the inequality $\alpha+1<p(b+1)$,
\begin{equation}\label{first condition of Lemma Integral-[x,y] }
b+Ap'=b-\frac{1+\alpha}{p}+\frac{\varepsilon q}{p-q}>-1+\frac{\varepsilon q}{p-q}>-1.
\end{equation}
By (\ref{new form solution-les for A,B}) again,
\begin{equation}\label{second condition of Lemma Integral-[x,y] }
Bq+\beta=(q-1)\left(-\frac{1+\beta}{q}+\frac{\varepsilon p}{p-q}\right)+\beta=-1+\frac{1+\beta}{q}+\frac{\varepsilon p(q-1)}{p-q}>-1.
\end{equation}
By (\ref{ value of c}) and (\ref{first condition of Lemma Integral-[x,y] }),
\begin{equation}\label{third condition of Lemma Integral-[x,y] }
c-(b+Ap')=-\varepsilon+\frac{1+\beta}{q}-\frac{\varepsilon q}{p-q}=\frac{1+\beta}{q}-\frac{\varepsilon p}{p-q}>0
\end{equation}
provided $\varepsilon<(\frac{1}{q}-\frac{1}{p})(1+\beta)$. Lastly, by (\ref{ value of c}), (\ref{second condition of Lemma Integral-[x,y] }), and  the inequality $\alpha+1<p(b+1)$,
\begin{equation}\label{fourth condition of Lemma Integral-[x,y] }
\begin{split}
c-(Bq+\beta)&=b+\frac{1+\beta}{q}-\frac{1+\alpha}{p}-\varepsilon+1-\frac{1+\beta}{q}-\frac{\varepsilon p(q-1)}{p-q}\\
&=b+1-\frac{1+\alpha}{p}-\frac{\varepsilon q(p-1)}{p-q}>0
\end{split}
\end{equation}
provided $\varepsilon<\frac{p}{p-1}(\frac{1}{q}-\frac{1}{p})(b+1-\frac{1+\alpha}{p})$. Finally, we verify the third condition of Theorem \ref{Schur test 2}, that is the finiteness of the double integral
\begin{equation*}
\int_{\mathbb{B}} \int_{\mathbb{B}} \frac{(1-|x|^2)^{(b-\alpha)}}{[x,y]^{(n+c)}}(1-|x|^2)^{Ap'}(1-|y|^2)^{Bq} d\nu_{\alpha}(x)d\nu_{\beta}(y).
\end{equation*}
We call it $\mathcal{I}$. We first estimate the integral with respect to $d\nu(x)$ by Lemma \ref{Integral-[x,y]} and obtain
\begin{equation*}
\mathcal{I}\sim \int_{\mathbb{B}} (1-|y|^2)^{Bq+\beta-c+b+Ap'}d\nu(y)
\end{equation*}
by (\ref{third condition of Lemma Integral-[x,y] }). Moreover, by (\ref{second condition of Lemma Integral-[x,y] }) and (\ref{third condition of Lemma Integral-[x,y] }), the power of the $(1-|y|^2)$ is
\begin{align*}
Bq+\beta-(c-(b+Ap'))&=-1+\frac{1+\beta}{q}+\frac{\varepsilon p(q-1)}{p-q}-\frac{1+\beta}{q}+\frac{\varepsilon p}{p-q}\\
&=-1+\frac{\varepsilon pq}{p-q}>-1,
\end{align*}
and this makes $\mathcal{I}$ finite. Hence for
\begin{equation*}
0<\varepsilon< \min \left\{n-1,\left(\frac{1}{q}-\frac{1}{p}\right)(1+\beta),\frac{p}{p-1}\left(\frac{1}{q}-\frac{1}{p}\right)\left(b+1-\frac{1+\alpha}{p}\right)\right\},
\end{equation*}
Theorem \ref{Schur test 2} using the selected functions $\phi$ and $\psi$ with the powers in  (\ref{new form solution-les for A,B}) applies  and proves that $S_{bc}$ is bounded from $L^{p}_{\alpha}$ to $L^{q}_{\beta}$ with $1<q< p<\infty$ when the inequalities in (iii) hold.

Next, let $1=q<p<\infty$. Assume that $f\in L_{\alpha}^{p}$. Writing the $L_{\beta}^{1}$ norm of $S_{bc}f$ explicitly and  applying  Fubini's theorem, then applying the H\"{o}lder inequality, we obtain
\begin{align*}
\|S_{bc}f\|_{L_{\beta}^{1}}&\lesssim \int_{\mathbb{B}} \int_{\mathbb{B}} |R_{c}(x,y)||f(x)|(1-|x|^2)^{b}d\nu(x)(1-|y|^2)^{\beta}d\nu(y)\\
&\lesssim\int_{\mathbb{B}}|f(x)|(1-|x|^2)^{b}\int_{\mathbb{B}} \frac{(1-|y|^2)^{\beta}}{[x,y]^{n+c}}d\nu(y)d\nu(x)\\
&= \int_{\mathbb{B}}|f(x)|(1-|x|^2)^{\alpha/p}\int_{\mathbb{B}} \frac{(1-|y|^2)^{\beta}d\nu(y)}{[x,y]^{n+c}}(1-|x|^2)^{b-\alpha/p}d\nu(x)\\
&\lesssim \|f\|_{L_{\alpha}^{p}}\left(\int_{\mathbb{B}}\left(\int_{\mathbb{B}} \frac{(1-|y|^2)^{\beta}d\nu(y)}{[x,y]^{n+c}}\right)^{p'}(1-|x|^2)^{(b-\alpha/p)p'}d\nu(x)\right)^{1/p'}\\
&=:J^{1/p'}\|f\|_{L_{\alpha}^{p}}.
\end{align*}
We will show that $J$ is finite using Lemma \ref{Integral-[x,y]} since $\beta>-1$ as required.

 Firstly, if $c-\beta<0$, $J\sim \int_{\mathbb{B}}(1-|x|^2)^{(b-\alpha/p)p'}d\nu(x)$. By the first inequality in (iii), we have
 \begin{equation*}
(b-\frac{\alpha}{p})p'=p'\big(b+1-\frac{1+\alpha}{p}-\frac{1}{p'}\big)>p'\big(-\frac{1}{p'}\big)=-1.
\end{equation*}
 Thus  $J$ is finite by Lemma \ref{an estimate from calculus}. Next, if $c-\beta=0$, then
 \begin{equation*}
J\sim \int_{\mathbb{B}}\left(1+\log \dfrac{1}{(1-|x|^{2})}\right)^{p'}(1-|x|^2)^{(b-\alpha/p)p'}d\nu(x)<\infty
\end{equation*}
also by Lemma \ref{an estimate from calculus} . Lastly, if $c-\beta>0$, then $J\sim \int_{\mathbb{B}}(1-|x|^2)^{(b-\alpha/p)p'-(c-\beta)p'}d\nu(x)$. But by the second inequality in (iii), we have
 \begin{equation*}
(b-\frac{\alpha}{p})p'-(c-\beta)p'=p'\big(b-c+1+\beta-\frac{1+\alpha}{p}-\frac{1}{p'}\big)>p'\big(-\frac{1}{p'}\big)=-1
\end{equation*}
 and thus $J$ is finite once again. Therefore $\|S_{bc}f\|_{L_{\beta}^{1}}\lesssim \|f\|_{\mathcal{L}_{\alpha}^{p}} $ and $S_{bc}$ is bounded from $L^{p}_{\alpha}$ to $L^{1}_{\beta}$ when $1=q<p<\infty$.
\end{proof}

\begin{proof}[Proof of sufficiency for Theorem \ref{Theorem-Boundedness of T,S 3.1}]\label{Sufficiency Proofs 3.1}
Let $f\in L_{\alpha}^{p}$. Writing  $S_{bc}f(y)$ explicitly and applying H\"{o}lder inequality with the measure $\nu_{\alpha}$ yields
\begin{align*}
(1-|y|^2)^{\beta}|Sf(y)|&\lesssim (1-|y|^2)^{\beta}\int_{\mathbb{B}} |R_{c}(x,y)||f(x)|(1-|x|^2)^{b}d\nu(x)\\
&\lesssim(1-|y|^2)^{\beta}\int_{\mathbb{B}}|f(x)|\frac{(1-|x|^2)^{b-\alpha}}{[x,y]^{n+c}}d\nu_{\alpha}(x)\\
&\lesssim\|f\|_{L_{\alpha}^{p}}(1-|y|^2)^{\beta}\left(\int_{\mathbb{B}}\frac{(1-|x|^2)^{(b-\alpha)p'+\alpha}}{[x,y]^{(n+c)p'}}d\nu(x)\right)^{1/p'}\\
&=:J(y)\|f\|_{L_{\alpha}^{p}}.
\end{align*}
 We will show that $J$ is bounded on $\mathbb{B}$ by using Lemma \ref{Integral-[x,y]}. First, notice that
\begin{equation*}
(b-\alpha)p'+\alpha+1=p'\left(b-\alpha+\frac{(\alpha+1)(p-1)}{p}\right)=p'\left(1+b-\frac{\alpha+1}{p}\right)>0
\end{equation*}
by the the first inequality of (iii) as required.
Consider that
\begin{align*}
\rho&=(n+c)p'-n-(b-\alpha)p'-\alpha=p'\left(n+c-b+\alpha-\frac{n+\alpha}{p'}\right)\\
&=p'\left(c-b+\frac{n+\alpha}{p}\right).
\end{align*}
If $\rho<0$, then the integral in $J(y)$ is bounded and $J(y)$ is also bounded for all $y\in \mathbb{B}$ since $\beta\geq 0$. Note that this is obvious from the second inequality in (iii) when $\beta=0$. Next, if $\rho=0$, then the integral in $J(y)$ is $(1+\log (1-|y|^{2})^{-1})^{1/p'}$ and since (iii) reads $\beta>0$, and therefore $J(y)$ is bounded for all $y\in \mathbb{B}$ by (\ref{obvious-inequalities}). Lastly, if $\rho>0$, then $J(y)\sim (1-|y|^2)^{\beta-\rho/p'}$. But by the second inequality in (iii), we have  $\beta-\rho/p'=\beta-c+b-\frac{n+\alpha}{p}\geq 0$ and thus $J(y)$ is  bounded for all $y\in \mathbb{B}$ once again. Then  $(1-|y|^2)^{\beta} |S_{bc}f(y)|\lesssim \|f\|_{L_{\alpha}^{p}}$ for all $y\in \mathbb{B}$  and $\|S_{bc}f\|_{\mathcal{L}_{\beta}^{\infty}}\lesssim \|f\|_{L_{\alpha}^{p}}$.
 Thus $S_{bc}$ is bounded from $L^{p}_{\alpha}$ to $\mathcal{L}^{\infty}_{\beta}$ with  $1
 <p<q=\infty$.
\end{proof}

\begin{proof}[Proof of sufficiency for Theorem \ref{Theorem-Boundedness of T,S 3.2}]\label{Sufficiency Proofs 3.2}
Let $f\in L_{\alpha}^{1}$. If $\alpha=b$ and $\beta=0$, then $c<-n$  by the second inequality in (iii) and $|R_{c}(x,y)|$ is bounded by Lemma \ref{Lemma-Kernel-Estimate}. So we have
\begin{equation*}
|S_{bc}f(y)|\lesssim \int_{\mathbb{B}}|R_{c}(x,y)||f(x)|(1-|x|^2)^{b}d\nu(x)\lesssim \|f\|_{L_{\alpha}^{1}} \quad (y\in \mathbb{B}).
\end{equation*}
Then $\|S_{bc}f\|_{L^{\infty}}\lesssim \|f\|_{L_{\alpha}^{1}}$.

Otherwise $\alpha\leq b$ and $\beta>0$, and there are values of $c>-n$ satisfying the inequalities in (iii). So in the rest of proof we can assume $c>-n$ by Corollary \ref{Corollary-Lemma-Kernel-comparasion}. Then we write $S_{bc}f(y)$ explicitly, and obtain
\begin{align*}
(1-|y|^2)^{\beta}|S_{bc}f(y)|&\leq (1-|y|^2)^{\beta}\int_{\mathbb{B}} |R_{c}(x,y)||f(x)|(1-|x|^2)^{b}d\nu(x)\\
&\lesssim (1-|y|^2)^{\beta}\int_{\mathbb{B}}|f(x)|\frac{(1-|x|^2)^{b-\alpha}}{[x,y]^{n+c}}d\nu_{\alpha}(x)\\
&=\int_{\mathbb{B}}|f(x)|(1-|y|^2)^{\beta}\frac{(1-|x|^2)^{b-\alpha}}{[x,y]^{n+c}}d\nu_{\alpha}(x)\\
&=:\int_{\mathbb{B}}|f(x)|J(x,y)d\nu_{\alpha}(x).
\end{align*}
Since $[x,y]\gtrsim(1-|x|^2) $ and $[x,y]\gtrsim (1-|y|^2) $ for $x,y \in \mathbb{B}$, we have $J(x,y)\lesssim (1-|x|^2)^{b-\alpha+\beta-(n+c)}$ for all such $y$. Note that the power here is nonnegative by the second inequality in (iii) yielding that $J(x,y)$ is bounded for all $x,y \in \mathbb{B}$. So we get that
\begin{equation*}
(1-|x|^2)^{\beta}|S_{bc}f(y)|\lesssim \int_{\mathbb{B}}|f(x)|d\nu_{\alpha}(x)=\|f\|_{L_{\alpha}^{1}} \quad (y\in \mathbb{B})
\end{equation*}
 and $\|S_{bc}f\|_{\mathcal{L}_{\beta}^{\infty}}\lesssim \|f\|_{L_{\alpha}^{1}}$.
 Thus $S_{bc}$ is bounded from $L^{1}_{\alpha}$ to $\mathcal{L}^{\infty}_{\beta}$.
\end{proof}

\begin{proof}[Proof of sufficiency for Theorem \ref{Theorem-Boundedness of T,S 3.3}]\label{Sufficiency Proofs 3.3}
First, let $q=1$. Assume that $f\in \mathcal{L}_{\alpha}^{\infty}$. Writing the $L_{\beta}^{1}$ norm of $S_{bc}f$ explicitly and  applying  Fubini's theorem, taking the $\mathcal{L}_{\alpha}^{\infty}$ norm of $f$ out of integral, we obtain
\begin{align*}
\|S_{bc}f\|_{L_{\beta}^{1}}&\lesssim \int_{\mathbb{B}} \int_{\mathbb{B}} |R_{c}(x,y)||f(x)|(1-|x|^2)^{b}d\nu(x)(1-|y|^2)^{\beta}d\nu(y)\\
&\lesssim \int_{\mathbb{B}}|f(x)|(1-|x|^2)^{b}\int_{\mathbb{B}} \frac{(1-|y|^2)^{\beta}}{[x,y]^{n+c}}d\nu(y)d\nu(x)\\
&\lesssim \|f\|_{L_{\alpha}^{\infty}}\int_{\mathbb{B}}(1-|x|^2)^{b-\alpha}\int_{\mathbb{B}} \frac{(1-|y|^2)^{\beta}}{[x,y]^{n+c}}d\nu(y)d\nu(x)
=:J\|f\|_{L_{\alpha}^{\infty}}.
\end{align*}
We will show that $J$ is finite using Lemma \ref{Integral-[x,y]} since $\beta>-1$ as required.

Firstly, if $c-\beta<0$, then the inner integral in $J$ is bounded and $J$ is finite since $b-\alpha>-1$ by the first inequality of (iii). Next, if $c-\beta=0$, then the inner integral is $1+\log (1/(1-|x|^{2}))^{-1}$. Then $J(x)$ is finite by $b-\alpha>-1$   and Lemma \ref{an estimate from calculus}. Lastly, if $c-\beta>0$, then $J\sim \int_{\mathbb{B}}(1-|x|^2)^{b-\alpha-c+\beta}d\nu(x)$. But by the second inequality in (iii) we have  $b-\alpha-c+\beta>-1$ and thus $J$ is finite once again. Therefore $\|S_{bc}f\|_{L_{\beta}^{1}}\lesssim \|f\|_{\mathcal{L}_{\alpha}^{\infty}}$ and $S_{bc}$ is bounded from $\mathcal{L}^{\infty}_{\alpha}$ to $L^{1}_{\beta}$.

Next, let $1<q<p=\infty$. Assume that $f\in \mathcal{L}_{\alpha}^{\infty}$. Writing $L_{\beta}^{q}$ norm of $S_{bc}f$ explicitly and taking the $\mathcal{L}_{\alpha}^{\infty}$ norm of $f$ out of integral, we obtain
\begin{align*}
\|Sf\|^{q}_{L_{\beta}^{q}}&= \int_{\mathbb{B}} \left|\int_{\mathbb{B}} |R_{c}(x,y)||f(x)|(1-|x|^2)^{b}d\nu(x)\right|^{q}d\nu_{\beta}(y)\\
&\lesssim \|f\|^{q}_{\mathcal{L}_{\alpha}^{\infty}}\int_{\mathbb{B}}(1-|y|^2)^{\beta}\left(\int_{\mathbb{B}} \frac{(1-|x|^2)^{b-\alpha}}{[x,y]^{n+c}}d\nu(x)\right)^{q}d\nu(y)
=:J\|f\|^{q}_{\mathcal{L}_{\alpha}^{\infty}}.
\end{align*}
We will show that $J$ is finite using Lemma \ref{Integral-[x,y]} since $\beta-\alpha>-1$ by the first inequality in (iii)  as required.

 Firstly, if $c-b+\alpha<0$, then the inner integral in $J$ is bounded and $J$ is finite since $\beta>-1$ . Next, if $c-b+\alpha=0$, then the inner integral is $1+\log (1/(1-|x|^{2}))^{-1}$. Then $J$ is finite by  Lemma \ref{an estimate from calculus}. Lastly, if $c-b+\alpha>0$, then $J\sim \int_{\mathbb{B}}(1-|y|^2)^{\beta-(c-b+\alpha)q}d\nu(y)$. But  we have
\begin{equation*}
\beta-(c-b+\alpha)q=\beta-(c-b+\alpha)q+1-1=q(\frac{\beta+1}{q}-c+b-\alpha)-1>-1
\end{equation*}
  by the second inequality in (iii)  and thus $J$ is finite once again. Therefore $\|S_{bc}f\|^{q}_{L_{\beta}^{q}}\- \lesssim \|f\|^{q}_{\mathcal{L}_{\alpha}^{\infty}}$ and $S_{bc}$ is bounded from $\mathcal{L}^{\infty}_{\alpha}$ to $L^{q}_{\beta}$.
\end{proof}

\begin{proof}[Proof of sufficiency for Theorem \ref{Theorem-Boundedness of T,S 3.4}]\label{Sufficiency Proofs 3.4}
Let $f\in \mathcal{L}_{\alpha}^{\infty}$. Writing  $S_{bc}f(y)$ explicitly and taking the $\mathcal{L}_{\alpha}^{\infty}$ norm of $f$ out of integral, we obtain
\begin{align*}
(1-|y|^2)^{\beta}|Sf(y)|&\lesssim (1-|y|^2)^{\beta} \int_{\mathbb{B}} |R_{c}(x,y)||f(x)|(1-|x|^2)^{b}d\nu(x)\\
&\lesssim\|f\|_{L_{\alpha}^{\infty}} (1-|y|^2)^{\beta}\int_{\mathbb{B}} \frac{(1-|x|^2)^{b-\alpha}}{[x,y]^{n+c}}d\nu(x)\\
&=:J(y)\|f\|_{L_{\alpha}^{\infty}}.
\end{align*}
We will show that $J$ is bounded on $\mathbb{B}$ by using Lemma \ref{Integral-[x,y]} since $b-\alpha>-1$ by the first inequality in (iii)  as required.

Firstly, if $c-b+\alpha<0$, then the  integral in $J$ is bounded and $J$ is bounded since $\beta>-1$ . Next, if $c-b+\alpha=0$, then the inner integral is $1+\log (1/(1-|x|^{2}))^{-1}$. Then $J$ is finite by  (\ref{obvious-inequalities}). Lastly, if $c-b+\alpha>0$, then $J(y)\sim (1-|y|^2)^{\beta-c+b-\alpha}$. But  we have $\beta-c+b-\alpha\geq 0$
  by the second inequality in (iii) and thus $J$ is bounded once again. Therefore $\|S_{bc}f\|_{\mathcal{L}_{\beta}^{\infty}}\lesssim \|f\|_{\mathcal{L}_{\alpha}^{\infty}}$ and $S_{bc}$ is bounded from $\mathcal{L}^{\infty}_{\alpha}$ to $\mathcal{L}^{\infty}_{\beta}$. This completes the proof.
\end{proof}

\bibliographystyle{amsalpha}

\end{document}